\documentclass[numbers=enddot,12pt,final,onecolumn,notitlepage]{scrartcl}%
\usepackage[all,cmtip]{xy}
\usepackage{amsfonts}
\usepackage{amssymb}
\usepackage{framed}
\usepackage{amsmath}
\usepackage{comment}
\usepackage{needspace}
\usepackage{color}
\usepackage{hyperref}
\usepackage[sc]{mathpazo}
\usepackage[T1]{fontenc}
\usepackage{amsthm}
\usepackage{fancyhdr}
\usepackage{ytableau}
\usepackage{tabu}
\providecommand{\U}[1]{\protect\rule{.1in}{.1in}}
\theoremstyle{definition}
\newtheorem{theo}{Theorem}[section]
\newenvironment{theorem}[1][]
{\begin{theo}[#1]\begin{leftbar}}
{\end{leftbar}\end{theo}}
\newtheorem{lem}[theo]{Lemma}
\newenvironment{lemma}[1][]
{\begin{lem}[#1]\begin{leftbar}}
{\end{leftbar}\end{lem}}
\newtheorem{prop}[theo]{Proposition}
\newenvironment{proposition}[1][]
{\begin{prop}[#1]\begin{leftbar}}
{\end{leftbar}\end{prop}}
\newtheorem{defi}[theo]{Definition}
\newenvironment{definition}[1][]
{\begin{defi}[#1]\begin{leftbar}}
{\end{leftbar}\end{defi}}
\newtheorem{remk}[theo]{Remark}
\newenvironment{remark}[1][]
{\begin{remk}[#1]\begin{leftbar}}
{\end{leftbar}\end{remk}}
\newtheorem{coro}[theo]{Corollary}
\newenvironment{corollary}[1][]
{\begin{coro}[#1]\begin{leftbar}}
{\end{leftbar}\end{coro}}
\newtheorem{conv}[theo]{Convention}

\newtheorem{warn}[theo]{Warning}

\newtheorem{conj}[theo]{Conjecture}

\newtheorem{exmp}[theo]{Example}
\newenvironment{example}[1][]
{\begin{exmp}[#1]\begin{leftbar}}
{\end{leftbar}\end{exmp}}

\newenvironment{todo}{}{}
\excludecomment{noncompile}
\excludecomment{obsolete}
\excludecomment{todo}

\newcommand{\Nplus}{\mathbb{N}_{+}}
\newcommand{\NN}{\mathbb{N}}
\let\sumnonlimits\sum
\let\prodnonlimits\prod
\renewcommand{\sum}{\sumnonlimits\limits}
\renewcommand{\prod}{\prodnonlimits\limits}
\def\seplist{{\operatorname{seplist}}} 
\def\ceq{{\operatorname{ceq}}}
\def\ircont{{\operatorname{ircont}}}
\def\cont{{\operatorname{cont}}}
\def\ceqvar{{{\alpha}}} 
\def\seplistvar{{{\nu}}} 
\def\supp{{\operatorname{supp}}}
\def\NS{{\operatorname{NR}}}
\def\g{{\widetilde{g}}}
\def\t{{\mathbf{t}}}
\def\x{{\mathbf{x}}}
\def\lm{{\lambda/\mu}}

\def\N{{\mathbb{N}}}
\def\Z{{\mathbb{Z}}}
\def\B{{\mathbf{B}}}

\def\OneTwoRPP{{\operatorname{RPP}^{12}\left(  \lambda/\mu\right)}}
\def\BenignTables{{\operatorname{BT}^{12}\left(  \lambda/\mu\right)}}
\def\OneTwoRPPCutvar{{\operatorname{RPP}^{12}\left(  \lambda/\mu ;\seplistvar \right)}}
\def\flip{{\operatorname{flip}}}

\begin{document}

\title{Refined dual stable Grothendieck polynomials and generalized Bender-Knuth involutions}
\author{Pavel Galashin, Darij Grinberg, and Gaku Liu}
\date{\today}
\maketitle

\begin{abstract}
The dual stable Grothendieck polynomials are a deformation of
the Schur functions, originating in the study of the $K$-theory of the
Grassmannian. We generalize these polynomials by introducing a
countable family of additional parameters, and we prove that this
generalization still defines symmetric functions. For this fact, we
give two self-contained proofs, one of which constructs a family of
involutions on the set of reverse plane partitions generalizing the
Bender-Knuth involutions on semistandard tableaux, whereas the other
classifies the structure of reverse plane partitions with entries $1$
and $2$.
\end{abstract}

\section{Introduction}

Thomas Lam and Pavlo Pylyavskyy, in \cite[\S 9.1]{LamPyl}, (and earlier Mark
Shimozono and Mike Zabrocki in unpublished work of 2003) studied \textit{dual
stable Grothendieck polynomials}, a deformation (in a sense) of the Schur
functions. Let us briefly recount their definition.

Let $\lambda/\mu$ be a skew partition. The Schur function $s_{\lambda/\mu}$ is
a multivariate generating function for the semistandard tableaux of shape
$\lambda/\mu$. In the same vein, the dual stable Grothendieck
polynomial 
$g_{\lambda/\mu}$ is
a generating function for the reverse plane partitions of shape $\lambda/\mu$;
these, unlike semistandard tableaux, are only required to have their entries
increase \textit{weakly} down columns (and along rows). More precisely,
$g_{\lambda/\mu}$ is a formal power series in countably many commuting
indeterminates $x_{1},x_{2},x_{3},\ldots$ defined by
\[
g_{\lambda/\mu}=\sum_{\substack{T\text{ is a reverse plane}\\\text{partition
of shape }\lambda/\mu}}\mathbf{x}^{\operatorname*{ircont}\left(  T\right)  },
\]
where $\mathbf{x}^{\operatorname*{ircont}\left(  T\right)  }$ is the monomial
$x_{1}^{a_{1}}x_{2}^{a_{2}}x_{3}^{a_{3}}\cdots$ whose $i$-th exponent $a_{i}$
is the number of \textit{columns} (rather than cells) 
of $T$ containing the entry $i$. As proven in
\cite[\S 9.1]{LamPyl}, this power series $g_{\lambda/\mu}$ is a symmetric
function (albeit, unlike $s_{\lambda/\mu}$, an inhomogeneous one in general).
Lam and Pylyavskyy connect the $g_{\lambda/\mu}$ to the (more familiar)
\textit{stable Grothendieck polynomials} $G_{\lambda/\mu}$ (via a duality
between the symmetric functions and their completion, which explains the name
of the $g_{\lambda/\mu}$; see \cite[\S 9.4]{LamPyl}) and to the $K$-theory of
Grassmannians (\cite[\S 9.5]{LamPyl}).

We devise a common generalization of the dual stable Grothendieck polynomial
$g_{\lambda/\mu}$ and the classical skew Schur function $s_{\lambda/\mu}$.
Namely, if $t_{1},t_{2},t_{3},\ldots$ are countably many indeterminates, then we
set%
\[
\widetilde{g}_{\lambda/\mu}=\sum_{\substack{T\text{ is a reverse
plane}\\\text{partition of shape }\lambda/\mu}}\mathbf{t}^{\operatorname*{ceq}%
\left(  T\right)  }\mathbf{x}^{\operatorname*{ircont}\left(  T\right)  },
\]
where $\mathbf{t}^{\operatorname*{ceq}\left(  T\right)  }$ is the product
$t_{1}^{b_{1}}t_{2}^{b_{2}}t_{3}^{b_{3}}\cdots$ whose $i$-th exponent $b_{i}$
is the number of cells in the $i$-th row of $T$ whose entry equals the entry
of their neighbor cell directly below them. This $\widetilde{g}_{\lambda/\mu}$
becomes $g_{\lambda/\mu}$ when all the $t_{i}$ are set to $1$, and becomes
$s_{\lambda/\mu}$ when all the $t_{i}$ are set to $0$.

Our main result, Theorem \ref{thm.gtilde.symm}, states that
$\widetilde{g}_{\lambda/\mu}$ is a symmetric function (in the $x_{1}%
,x_{2},x_{3},\ldots$).

We prove this result (thus obtaining a new proof of
\cite[Theorem 9.1]{LamPyl}) first using an elaborate generalization
of the classical Bender-Knuth involutions to reverse plane partitions,
and then for a second time by analyzing the structure of reverse plane
partitions whose entries lie in $\left\{1, 2\right\}$. The second
proof reflects back on the first, in particular providing an
alternative definition of the generalized Bender-Knuth involutions
constructed in the first proof, and showing that these involutions
are (in a sense) ``the only reasonable choice''.

The present paper is organized as follows: In Section \ref{sect.notations}, we
recall classical definitions and introduce notations pertaining to
combinatorics and symmetric functions. In Section \ref{sect.def}, we define
the refined dual stable Grothendieck polynomials $\widetilde{g}_{\lambda/\mu}$, state
our main result (that they are symmetric functions), and do the first steps of
its proof (by reducing it to a purely combinatorial statement about the
existence of an involution with certain properties). In Section \ref{sect.construction}, we describe the idea of constructing this involution in an elementary way without proofs. In Section
\ref{sect.proof}, we prove various properties of this involution advertised in Section \ref{sect.construction}, thus finishing the proof of our main result. In Section \ref{sect.BKclassical}, we
recapitulate the definition of the classical Bender-Knuth involution, and
show that our involution is a generalization of the latter.
Finally, in Section \ref{sect.structure} we study the structure of
reverse plane partitions with entries belonging to $\left\{1, 2\right\}$,
which (in particular) gives us an explicit formula for the
$\mathbf{t}$-coefficients of $\g_\lm(x_1,x_2,0,0,\dots;\t)$, and shines
a new light on the involution constructed in Sections \ref{sect.construction}
and \ref{sect.proof}
(also showing that it is the unique involution that shares certain natural
properties with the classical Bender-Knuth involutions).

\begin{todo}
\begin{itemize}
\item More reasons why the reader should
care about dual stable Grothendieck polynomials?

\item What I wrote about $K$-theory is rather shallow. More details?

More specifically, and interestingly, I am wondering if our $\widetilde{g}%
_{\lambda/\mu}$ aren't K-theoretical classes of something multigraded (toric
structure on the Grassmannian? there are two sides from which we can multiply
a matrix by a diagonal matrix, and even if we \textquotedblleft use
up\textquotedblright\ one for taking \textquotedblleft
characters\textquotedblright, the other one is still there).
\end{itemize}
\end{todo}

\subsection{Acknowledgments}

We owe our familiarity with dual stable Grothendieck polynomials to Richard
Stanley. We thank Alexander Postnikov for providing context and motivation,
and Thomas Lam and Pavlo Pylyavskyy for interesting conversations.

\begin{todo}
Keep this up to date.
\end{todo}

\section{\label{sect.notations}Notations and definitions}

Let us begin by defining our notations (including some standard conventions
from algebraic combinatorics).

\subsection{Partitions and tableaux}

We set $\mathbb{N}=\left\{  0,1,2,\ldots\right\}  $ and $\mathbb{N}%
_{+}=\left\{  1,2,3,\ldots\right\}  $.

 A sequence $\alpha=\left(\alpha_{1},\alpha_{2},\alpha_{3},\ldots\right)$ of nonnegative integers is called a \textit{weak composition} if the sum of its entries (denoted $\left\vert \alpha\right\vert$) is finite.
 We shall always write $\alpha_i$ for the $i$-th entry of a weak composition $\alpha$.

A \textit{partition} is a weak composition $\left(  \alpha_{1},\alpha
_{2},\alpha_{3},\ldots\right)  $ satisfying $\alpha_{1}\geq\alpha_{2}
\geq\alpha_{3}\geq\cdots$.
As usual, we often omit trailing zeroes when writing a partition (e.g.,
the partition $\left(5,2,1,0,0,0,\ldots\right)$ can thus be written as
$\left(5,2,1\right)$).

We identify each partition $\lambda$ with the subset
$\left\{ \left( i, j \right) \in \Nplus^2 \mid j \leq \lambda_i \right\}$
of $\Nplus^{2}$ (called \textit{the Young diagram of $\lambda$}).
We draw this subset as a Young diagram (which is a left-aligned table of
empty boxes, where the box $(1,1)$ is in the top-left corner while the
box $(2,1)$ is directly below it; this is the \textit{English notation},
also known as the \textit{matrix notation}); see \cite{Fulton97} for
the detailed definition.

\begin{todo}
What is the easiest place to refer the reader to for Young diagram basics,
which uses notations compatible with ours (such as "filling" and "skew
partition")?
\end{todo}

A \textit{skew partition} $\lambda/\mu$ is a pair $\left(\lambda, \mu\right)$ of partitions satisfying $\mu\subseteq\lambda$ (as subsets of the plane). In this case, we shall also often use the notation $\lambda/\mu$ for the set-theoretic difference of $\lambda$ and $\mu$.

 If $\lm$ is a skew partition, then a \textit{filling} of $\lm$ means a map $T:\lm\rightarrow\Nplus$. It is visually represented by drawing $\lm$ and filling each box $c$ with the entry $T(c)$. Three examples of a filling can be found on Figure \ref{fig:fillings}.

 A filling $T:\lm\rightarrow\Nplus$ of $\lm$ is called a \textit{reverse plane partition of shape $\lm$} if its values increase weakly in each row of $\lm$ from left to right and in each column of $\lm$ from top to bottom. If, in addition, the values of $T$ increase strictly down each column, then $T$ is called a \textit{semistandard tableau of shape $\lm$}. (See Fulton's \cite{Fulton97} for an exposition of properties and applications of semistandard tableaux\footnote{Fulton calls semistandard tableaux just ``tableaux'', but otherwise is consistent with most of our notation.}.) We denote the set of all reverse plane partitions of shape $\lm$ by $\operatorname{RPP}\left(  \lambda/\mu\right)$.  We abbreviate reverse plane partitions as \textit{rpps}. 
 
 Examples of an rpp, of a non-rpp and of a semistandard tableau can be found on Figure \ref{fig:fillings}.

\begin{figure}
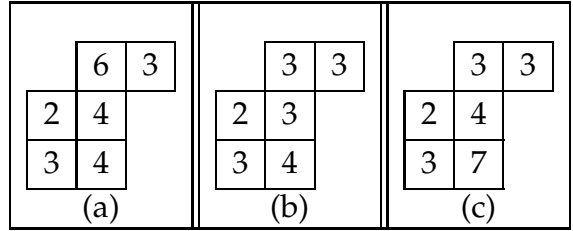

\begin{center}
 
\begin{tabular}{|c||c||c|}\hline
 & & \\
\begin{ytableau}
\none& 6 & 3\\
2 & 4 \\
3 & 4\\
\end{ytableau} &
\begin{ytableau}
\none& 3 & 3\\
2 & 3 \\
3 & 4\\
\end{ytableau} &
\begin{ytableau}
\none& 3 & 3\\
2 & 4 \\
3 & 7\\
\end{ytableau} \\
(a) & (b) & (c)\\
\hline
\end{tabular}\\
\caption{\label{fig:fillings} Fillings of $(3,2,2)/(1)$: (a) is not an rpp as it has a $4$ below a $6$, (b) is an rpp but not a semistandard tableau as it has a $3$ below a $3$, (c) is a semistandard tableau (and hence also an rpp).}
\end{center}

\end{figure}

\subsection{Symmetric functions}

A \textit{symmetric function} is defined to be a
bounded-degree\footnote{A power series is said to be \textit{bounded-degree}
if there is an $N \in \mathbb{N}$ such that only monomials of degree $\leq N$
appear in the series.}
power series in countably many indeterminates $x_1,x_2,x_3,\dots$
over $\Z$ that is invariant under (finite)
permutations\footnote{A permutation is \textit{finite} if it fixes all but
finitely many elements.} of $x_1,x_2,x_3,\dots$.

The symmetric functions form a ring, which is called the \textit{ring of
symmetric functions} and denoted by $\Lambda$.
(In \cite{LamPyl} this ring is denoted by $\operatorname*{Sym}$, while the notation $\Lambda$ is reserved for the
set of all partitions.) Much research has been done on symmetric functions and
their relations to Young diagrams and tableaux; see
\cite[Chapter 7]{Stan99}, \cite{Macdon95} and \cite[Chapter 2]{GriRei15} for expositions.

 Given a filling $T$ of a skew partition $\lm$, its \textit{content} is a weak composition $\operatorname*{cont}\left(  T\right)=\left(r_1,r_2,r_3,\dots\right)$, where $r_i=\left|T^{-1}(i)\right|$ is the number of entries of $T$ equal to $i$. For a skew partition $\lambda/\mu$, we define the \textit{Schur function}
$s_{\lambda/\mu}$ to be the formal power series 
\[
s_\lm(x_1,x_2,\dots)
= \sum_{\substack{T\text{ is a semistandard}\\\text{tableau of shape } \lm}}
\mathbf{x}^{\operatorname*{cont}\left(  T\right)  }
\in \Z\left[\left[x_1,x_2,x_3,\ldots\right]\right] .
\]
Here, for every weak composition $\alpha = \left(\alpha_1, \alpha_2, \alpha_3, \ldots\right)$, we define a monomial $\x^\alpha$ to be $x_1^{\alpha_1} x_2^{\alpha_2} x_3^{\alpha_3} \cdots$.
These Schur functions are symmetric:

 
\begin{proposition}
\label{prop.schur.symm}We have $s_{\lambda/\mu}\in\Lambda$ for every skew
partition $\lambda/\mu$.
\end{proposition}

This result appears, e.g., in \cite[Theorem 7.10.2]{Stan99} and
\cite[Proposition 2.11]{GriRei15}; it is commonly proven bijectively using the
so-called \textit{Bender-Knuth involutions}. We shall recall the definitions
of these involutions in Section \ref{sect.BKclassical}.

 Replacing ``semistandard
tableau'' by ``rpp'' in the
definition of a Schur function in general gives a non-symmetric function. Nevertheless, Lam
and Pylyavskyy \cite[\S 9]{LamPyl} have been able to define
symmetric functions from rpps, albeit using a subtler construction
instead of the content $\operatorname{cont}\left(  T\right)$.

  Namely, for a filling $T$ of a skew partition $\lm$, we define its
\textit{irredundant content} (or, by way of abbreviation, its
\textit{ircont statistic})
as the weak composition $\operatorname*{ircont}\left(
T\right) = \left(r_1,r_2,r_3,\dots\right)$ where $r_i$ is the number of \emph{columns} (rather than cells) of $T$ that contain an entry equal to $i$. For instance, if $T_a$, $T_b$, and $T_c$ are the fillings from Figure \ref{fig:fillings}, then their irredundant contents are
\begin{align*}
\ircont(T_a)=(0,1,2,1,0,1),\ \ircont(T_b)=(0,1,3,1),\ \ircont(T_c)=(0,1,3,1,0,0,1)
\end{align*}
(where we omit trailing zeroes),
because, for example, $T_a$ has one column with a $4$ in it (so $(\ircont(T_a))_4=1$) and $T_b$ contains three columns with a $3$ (so $(\ircont(T_b))_3=3$).

Notice that if $T$ is a semistandard tableau, then $\cont(T)$ and $\ircont(T)$ coincide.

For the rest of this section, we fix a skew partition $\lambda/\mu$. Now, the
\textit{dual stable Grothendieck polynomial} $g_{\lambda/\mu}$ is defined to
be the formal power series%
\[
\sum_{\substack{T\text{ is an rpp}\\\text{of shape }\lm}}\mathbf{x}^{\operatorname*{ircont}\left(  T\right)  }.
\]
Unlike the Schur function $s_{\lambda/\mu}$, it is (in
general) not homogeneous, because whenever a column of an rpp $T$ contains an
entry several times, the corresponding monomial $\mathbf{x}%
^{\operatorname*{ircont}\left(  T\right)  }$ \textquotedblleft
counts\textquotedblright\ this entry only once. It is fairly clear that the
highest-degree homogeneous component of $g_{\lambda/\mu}$ is $s_{\lambda/\mu}$
(the component of degree $\left\vert \lambda\right\vert -\left\vert
\mu\right\vert $). Therefore, $g_{\lambda/\mu}$ can be regarded as an
inhomogeneous deformation of the Schur function $s_{\lambda/\mu}$.

Lam and Pylyavskyy, in \cite[\S 9.1]{LamPyl}, have shown the following fact:

\begin{proposition}
\label{prop.g.symm}We have $g_{\lambda/\mu}\in\Lambda$ for every skew
partition $\lambda/\mu$.
\end{proposition}

They prove this proposition using generalized plactic algebras \cite[Lemma
3.1]{FomGre} (and also give a second, combinatorial proof for the case
$\mu=\varnothing$ by explicitly expanding $g_{\lambda/\varnothing}$ as a sum
of Schur functions).

In the next section, we shall introduce a refinement of these $g_{\lambda/\mu
}$, and later we will reprove Proposition \ref{prop.g.symm} in a
bijective 
and elementary way.

\section{\label{sect.def}Refined dual stable Grothendieck polynomials}

\subsection{Definition}

 Let $\t=\left(t_{1},t_{2},t_{3},\ldots\right)$ be a sequence of further indeterminates. For any weak composition $\alpha$, we define $\t^\alpha$ to be the monomial $t_1^{\alpha_1} t_2^{\alpha_2} t_3^{\alpha_3} \cdots$.

 If $T$ is a filling of a skew partition $\lm$,
then a \textit{redundant cell} of $T$ is a cell of $\lm$ whose entry is equal to the entry directly below it. That is, a cell $\left(  i,j\right)  $
of $\lm$ is redundant if $\left(  i+1,j\right)  $ is also a cell of $\lm$ and 
$T\left(  i,j\right)  =T\left(  i+1,j\right)  $. Notice that a semistandard
tableau is the same thing as an rpp which has no redundant
cells.

 If $T$ is a filling of $\lm$,
then we define the \textit{column equalities vector} (or,
by way of abbreviation, the \textit{ceq statistic})
of $T$ to be the weak composition
$\operatorname{ceq}\left(  T\right)=\left(c_1,c_2,c_3,\dots\right)$
where $c_i$ is the number of $j\in\mathbb{N}_{+}$ such that $\left(  i,j\right)$ is a redundant cell of $T$. Visually speaking, $\left(  \operatorname{ceq}\left(  T\right)  \right)
_{i}$ is the number of columns of $T$ whose entry in the $i$-th row equals
their entry in the $\left(  i+1\right)  $-th row. For instance, for fillings $T_a,$ $T_b,$ $T_c$ from Figure \ref{fig:fillings} we have $\ceq(T_a)=(0,1),\ \ceq(T_b)=(1),$ and $\ceq(T_c)=()$, where we again drop trailing zeroes.

Notice that $\left|\ceq(T)\right|$ is the number of redundant cells in $T$, so we have 
\begin{equation}\label{eq:sum.of.ceq.and.ircont}
 \left|\ceq(T)\right|+\left|\ircont(T)\right|=\left|\lm\right|
\end{equation}
 for all rpps $T$ of shape $\lm$.

Let now $\lambda/\mu$ be a skew partition. We set%
\[
\widetilde{g}_{\lambda/\mu}(\mathbf{x};\t)=\sum_{\substack{T\text{ is an rpp}\\\text{of shape
}\lm  }}\mathbf{t}^{\operatorname*{ceq}\left(
T\right)  }\mathbf{x}^{\operatorname*{ircont}\left(  T\right)  }
\in \Z\left[t_1, t_2, t_3, \ldots\right]\left[\left[x_1, x_2, x_3, \ldots\right]\right] .
\]

Let us give some examples of $\widetilde{g}_{\lambda/\mu}$.

\begin{example}
\label{exa.gtilde.1}

\begin{enumerate}

\item[\textbf{(a)}] If $\lm$ is a single row with $n$ cells, then for each rpp $T$ of shape $\lm$ we have $\ceq(T)=(0,0,\dots)$ and $\ircont(T)=\cont(T)$ (in fact, any rpp of shape $\lm$ is a semistandard tableau in this case). Therefore we get 
\[
\g_\lm(\x;\t)=h_n(\x)=\sum_{a_1\leq a_2\leq\dots\leq a_n} x_{a_1}x_{a_2}\cdots x_{a_n}.
\]
 Here $h_n(\x)$ is the $n$-th complete homogeneous symmetric function.

\item[\textbf{(b)}] If $\lm$ is a single column with $n$ cells, then, by (\ref{eq:sum.of.ceq.and.ircont}), for all rpps $T$ of shape $\lm$ we have $|\ceq(T)|+|\ircont(T)|=n$, so in this case
\[
\g_\lm(\x;\t)=\sum_{k=0}^{n}e_{k}\left(  t_{1},t_{2},\ldots,t_{n-1}\right)
e_{n-k}\left(  x_{1},x_{2},\ldots\right) =e_n(t_1,t_2,\dots,t_{n-1},x_1,x_2,\dots),
\]
where $e_{i}\left(  \xi_{1},\xi_{2},\xi_{3},\ldots\right)  $ denotes the
$i$-th elementary symmetric function in the indeterminates $\xi_{1},\xi
_{2},\xi_{3},\ldots$.

\end{enumerate}

\end{example}

The power series $\widetilde{g}_{\lambda/\mu}$ generalize the power series
$g_{\lambda/\mu}$ and $s_{\lambda/\mu}$ studied before. The following proposition is clear:

\begin{proposition}
\label{prop.gtilde.gener}Let $\lambda/\mu$ be a skew partition.

\begin{enumerate}
\item[\textbf{(a)}] Specifying $\t=\left(
1,1,1,\ldots\right)  $ yields $\widetilde{g}_{\lambda/\mu}(\x;\t)=g_{\lambda/\mu}(\x)$.

\item[\textbf{(b)}] Specifying $\t =\left(
0,0,0,\ldots\right)  $ yields $\widetilde{g}_{\lambda/\mu}(\x;\t)=s_{\lambda/\mu}(\x)$.

\end{enumerate}
\end{proposition} \qed

\subsection{The symmetry statement}

Our main result is now the following:

\begin{theorem}
\label{thm.gtilde.symm}Let $\lambda/\mu$ be a skew partition. Then
$\widetilde{g}_{\lambda/\mu}(\x;\t)$ is symmetric in $\x$.
\end{theorem}

Here, ``symmetric in $\x$'' means ``invariant under all finite
permutations of the indeterminates $x_1, x_2, x_3, \ldots$''
(while $t_1, t_2, t_3, \ldots$ remain unchanged).

Clearly, Theorem~\ref{thm.gtilde.symm} implies the symmetry of
$g_\lm$ and $s_\lm$ due to Proposition~\ref{prop.gtilde.gener}.

We shall prove Theorem \ref{thm.gtilde.symm} bijectively. The core of our
proof will be the following restatement of Theorem \ref{thm.gtilde.symm}:

\begin{theorem}
\label{thm.BK}Let $\lambda/\mu$ be a skew partition and let $i\in\mathbb{N}_{+}$. Then, there exists an
involution $\mathbf{B}_{i}:\operatorname{RPP}\left(  \lambda/\mu\right)
\rightarrow\operatorname{RPP}\left(  \lambda/\mu\right)  $ which preserves the ceq statistics and acts on the ircont statistic by the transposition of its $i$-th and $i+1$-th entries.
\end{theorem}

This involution $\mathbf{B}_{i}$ is a generalization of the $i$-th
Bender-Knuth involution defined for semistandard tableaux (see, e.g.,
\cite[proof of Proposition 2.11]{GriRei15}), but its definition is more
complicated than that of the latter.\footnote{We will compare our involution
$\mathbf{B}_{i}$ with the $i$-th Bender-Knuth involution in Section
\ref{sect.BKclassical}.} Defining it and proving its properties will take a
significant part of this paper.

\begin{proof}[Proof of Theorem~\ref{thm.gtilde.symm} using
Theorem~\ref{thm.BK}.]
We need to prove that $\widetilde{g}_{\lambda/\mu}(\x;\t)$ is invariant
under all finite permutations of the indeterminates
$x_1, x_2, x_3, \ldots$. The group of such permutations is generated by
$s_1, s_2, s_3, \ldots$, where for each $i \in \Nplus$, we define $s_i$
as the permutation of $\Nplus$ which transposes $i$ with $i+1$ and
leaves all other positive integers unchanged. Hence, it suffices to
show that $\widetilde{g}_{\lambda/\mu}(\x;\t)$ is invariant under each
of the permutations $s_1, s_2, s_3, \ldots$. In other words, it suffices
to show that $s_i \cdot \widetilde{g}_{\lambda/\mu}(\x;\t)
= \widetilde{g}_{\lambda/\mu}(\x;\t)$ for each $i \in \Nplus$.

So fix $i \in \Nplus$. In order to prove
$s_i \cdot \widetilde{g}_{\lambda/\mu}(\x;\t)
= \widetilde{g}_{\lambda/\mu}(\x;\t)$, it suffices to find a bijection
$\mathbf{B}_{i}:\operatorname{RPP}\left(  \lambda/\mu\right)
\rightarrow\operatorname{RPP}\left(  \lambda/\mu\right)  $ with the
property that every $T \in \operatorname{RPP}\left(  \lambda/\mu\right)$
satisfies $\ceq\left(\mathbf{B}_i\left(T\right)\right) = \ceq \left(T\right)$ and
$\ircont\left(\mathbf{B}_i\left(T\right)\right) = s_i \cdot \ircont \left(T\right)$.
Theorem~\ref{thm.BK} yields precisely such a bijection (even an
involution).
\end{proof}

\subsection{Reduction to 12-rpps}

Fix a skew partition $\lm$. We shall make one further simplification before we step to the actual proof of
Theorem \ref{thm.BK}. We define a \textit{12-rpp} to be an rpp whose entries all belong to the set $\left\{  1,2\right\}  $. We let $\OneTwoRPP$ be the set of all 12-rpps of shape $\lm$.

\begin{lemma}
\label{lem.BK} There exists an
involution $\mathbf{B}:\OneTwoRPP\rightarrow\OneTwoRPP$
which preserves the ceq statistic and switches the number of columns containing a $1$ with the number of columns containing a $2$ (that is, switches the first two entries of the ircont statistic).
\end{lemma}

 This Lemma implies Theorem \ref{thm.BK}: for any $i\in\Nplus$ and for $T$ an rpp of shape $\lm$, we construct $\B_i(T)$ as follows: 
 \begin{itemize}
  \item Ignore all entries of $T$ not equal to $i$ or $i+1$. 
  \item Replace all occurrences of $i$ by $1$ and all occurrences of $i+1$ by $2$. We get a 12-rpp $T'$ of some smaller shape (which is still a skew partition\footnote{Fine print: It has the form $\lm$ for some skew partition $\lm$, but this skew partition $\lm$ is not always uniquely determined (e.g., $\left(3,1,1\right)/\left(2,1\right)$ and $\left(3,2,1\right)/\left(2,2\right)$ have the same Young diagram). But the involution $\mathbf{B}$ constructed in the proof of Lemma~\ref{lem.BK} depends only on the Young diagram of $\lm$, and thus the choice of $\lm$ does not matter.}).
  \item Replace $T'$ by $\B(T')$.
  \item In $\B(T')$, replace back all occurrences of $1$ by $i$ and all occurrences of $2$ by $i+1$.
  \item Finally, restore the remaining entries of $T$ that were ignored on the first step.
 \end{itemize}
 
 It is clear that this operation acts on $\ircont(T)$ by a transposition of the $i$-th and $i+1$-th entries. The fact that it does not change $\ceq(T)$ is also not hard to show: the set of redundant cells remains the same.

\def\B{{\mathbf{B}}}

\section{Construction of $\mathbf{B}$\label{sect.construction}}

In this section we are going to sketch the definition of $\mathbf{B}$ and state some of its properties. We postpone the proofs until the next section.

For the whole Sections \ref{sect.construction} and \ref{sect.proof},
we shall be working in the situation of Lemma \ref{lem.BK}. In
particular, we fix a skew partition $\lm$.

A \textit{12-table} means a filling $T:\lm\rightarrow\left\{  1,2\right\}  $
of $\lm$
such that the entries of $T$ are weakly increasing down columns. (We do not
require them to be weakly increasing along rows.) Every column of a 12-table
is a sequence of the form $(1,1,\ldots,1,2,2,\ldots,2)$. We say that such a sequence is

\begin{itemize}
\item \textit{1-pure} if it is nonempty and consists purely of $1$'s,

\item \textit{2-pure} if it is nonempty and consists purely of $2$'s,

\item \textit{mixed} if it contains both $1$'s and $2$'s.
\end{itemize}

\def\flip{{\operatorname*{flip}}}

\begin{definition}
 \label{defi.flip}
For a 12-table $T$, we define $\flip(T)$ to be the 12-table obtained from $T$ by changing each column of $T$ as follows:

\begin{itemize}
\item \textbf{If} this column is 1-pure, we replace all its entries by $2$'s
(so that it becomes 2-pure).

\textbf{Otherwise}, if this column is 2-pure, we replace all its entries by
$1$'s (so that it becomes 1-pure).

\textbf{Otherwise} (i.e., if this column is mixed or empty), we do not change it.
\end{itemize}

\end{definition}

If $T$ is a 12-rpp then $\flip(T)$ need not be a 12-rpp, because it can contain a $2$ to the left of a $1$ in some row. We say that a positive integer $k$ is a \textit{descent} of a 12-table $P$ if there is a $2$ in the column $k$ and there is a $1$ to the right of it in the column $k+1$. We will encounter three possible kinds of descents depending on the types of columns $k$ and $k+1$:

\begin{itemize}
\item[(M1)] The $k$-th column of $P$ is mixed and the $\left(  k+1\right)  $-th column of $P$ is 1-pure.

\item[(2M)] The $k$-th column of $P$ is 2-pure and the $\left(  k+1\right)  $-th column of $P$ is mixed.

\item[(21)] The $k$-th column of $P$ is 2-pure and the $\left(  k+1\right)  $-th column of $P$ is 1-pure.
\end{itemize}

For an arbitrary 12-table it can happen also that two mixed columns form a descent, but such a descent will never arise in our process. 

For each of the three types of descents, we will define what it means to \textit{resolve} this descent. This is an operation which transforms the 12-table $P$ by changing the entries in its $k$-th and $\left(k+1\right)$-th columns. These changes can be informally explained by Figure \ref{fig:des-res-preliminary}:

\begin{figure}[here]
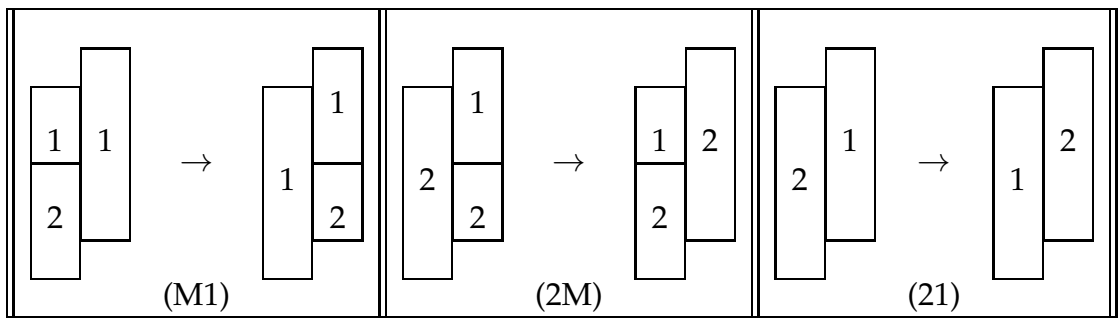

 
\begin{tabular}{||ccc||ccc||ccc||}\hline
 & & & & & & & & \\
\begin{tabular}{cc}
\cline{2-2} & \multicolumn{1}{|c|}{}\\
\cline{1-1} \multicolumn{1}{|c|}{} & \multicolumn{1}{|c|}{} \\
\multicolumn{1}{|c|}{1} & \multicolumn{1}{|c|}{1}\\
\cline{1-1} \multicolumn{1}{|c|}{} & \multicolumn{1}{|c|}{}\\
\multicolumn{1}{|c|}{2} & \multicolumn{1}{|c|}{}\\
\cline{2-2} \multicolumn{1}{|c|}{} \\
\cline{1-1}
\end{tabular}
&
$\rightarrow$
&
\begin{tabular}{cc}
\cline{2-2} & \multicolumn{1}{|c|}{}\\
\cline{1-1} \multicolumn{1}{|c|}{} & \multicolumn{1}{|c|}{1} \\
\multicolumn{1}{|c|}{} & \multicolumn{1}{|c|}{}\\
\cline{2-2} \multicolumn{1}{|c|}{1} & \multicolumn{1}{|c|}{}\\
\multicolumn{1}{|c|}{} & \multicolumn{1}{|c|}{2}\\
\cline{2-2} \multicolumn{1}{|c|}{} \\
\cline{1-1}
\end{tabular}
&
\begin{tabular}{cc}
\cline{2-2} & \multicolumn{1}{|c|}{}\\
\cline{1-1} \multicolumn{1}{|c|}{} & \multicolumn{1}{|c|}{1} \\
\multicolumn{1}{|c|}{} & \multicolumn{1}{|c|}{}\\
\cline{2-2} \multicolumn{1}{|c|}{2} & \multicolumn{1}{|c|}{}\\
\multicolumn{1}{|c|}{} & \multicolumn{1}{|c|}{2}\\
\cline{2-2} \multicolumn{1}{|c|}{} \\
\cline{1-1}
\end{tabular}
&
$\rightarrow$
&
\begin{tabular}{cc}
\cline{2-2} & \multicolumn{1}{|c|}{}\\
\cline{1-1} \multicolumn{1}{|c|}{} & \multicolumn{1}{|c|}{} \\
\multicolumn{1}{|c|}{1} & \multicolumn{1}{|c|}{2}\\
\cline{1-1} \multicolumn{1}{|c|}{} & \multicolumn{1}{|c|}{}\\
\multicolumn{1}{|c|}{2} & \multicolumn{1}{|c|}{}\\
\cline{2-2} \multicolumn{1}{|c|}{} \\
\cline{1-1}
\end{tabular}
&
\begin{tabular}{cc}
\cline{2-2} & \multicolumn{1}{|c|}{}\\
\cline{1-1} \multicolumn{1}{|c|}{} & \multicolumn{1}{|c|}{} \\
\multicolumn{1}{|c|}{} & \multicolumn{1}{|c|}{1}\\
\multicolumn{1}{|c|}{2} & \multicolumn{1}{|c|}{}\\
\multicolumn{1}{|c|}{} & \multicolumn{1}{|c|}{}\\
\cline{2-2} \multicolumn{1}{|c|}{} \\
\cline{1-1}
\end{tabular}
&
$\rightarrow$
&
\begin{tabular}{cc}
\cline{2-2} & \multicolumn{1}{|c|}{}\\
\cline{1-1} \multicolumn{1}{|c|}{} & \multicolumn{1}{|c|}{} \\
\multicolumn{1}{|c|}{} & \multicolumn{1}{|c|}{2}\\
\multicolumn{1}{|c|}{1} & \multicolumn{1}{|c|}{}\\
\multicolumn{1}{|c|}{} & \multicolumn{1}{|c|}{}\\
\cline{2-2} \multicolumn{1}{|c|}{} \\
\cline{1-1}
\end{tabular}\\
& (M1) & & & (2M) & & & (21) & \\\hline
\end{tabular}

\caption{The three descent-resolution steps \label{fig:des-res-preliminary}}
\end{figure}

\def\resk{{\operatorname{res}_{k}}}

For example, if $k$ is a descent of type (M1) in a 12-table $P$, then we define the 12-table $\resk P$ as follows: the $k$-th column of $\resk P$ is 1-pure; the $\left(  k+1\right)  $-th column of $\resk P$ is mixed and the highest $2$ in it is in the same row as the highest $2$ in the $k$-th column of $P$; all other columns of $\resk P$ are
copied over from $P$ unchanged. The definitions of $\resk P$ for the other two types of descents are similar (and will be elaborated upon in Subsection~\ref{subsect.resolving}). We say that $\resk P$ is obtained from $P$ by \textit{resolving} the descent $k$, and we say that passing from $P$ to $\resk P$ constitutes a \textit{descent-resolution step}. (Of course, a 12-table $P$ can have several descents and thus offer several ways to proceed by descent-resolution steps.)

Now the map $\B$ is defined as follows: take any 12-rpp $T$ and apply flip to it to get a 12-table $\flip(T)$. Next, apply descent-resolution steps to $\flip(T)$ in arbitrary order until we get a 12-table with no descents left. Put $\B(T):=P$. (A rigorous statement of this is Definition \ref{defi.B}.)

In the next section we will see that $\B(T)$ is well-defined (that is, the process terminates after a finite number of descent-resolution steps, and the result does not depend on the order of steps). We will also see that $\B$ is an involution $\OneTwoRPP\to\OneTwoRPP$ that satisfies the claims of Lemma \ref{lem.BK}. An alternative proof of all these facts can be found in Section \ref{sect.structure}.


\section{\label{sect.proof}Proof of Lemma \ref{lem.BK}}

We shall now prove Lemma \ref{lem.BK} in detail.

Recall that every column of a 12-table
is a sequence of the form $(1,1,\ldots,1,2,2,\ldots,2)$. If $s$ is a sequence of the form $(1,1,\ldots,1,2,2,\ldots,2)$, then we define the
\textit{signature} $\operatorname*{sig}\left(  s\right)
$ of $s$ to be
\[
\operatorname*{sig}\left(  s\right) = \left\{
\begin{array}
[c]{l}%
0,\text{ if }s\text{ is 2-pure or empty;}\\
1,\text{ if }s\text{ is mixed;}\\
2,\text{ if }s\text{ is 1-pure}%
\end{array}
\right.
.
\]
\begin{definition}
 \label{defi.fourtypes}
For any 12-table $T$, we define a nonnegative integer $\ell\left(
T\right)  $ by%
\[
\ell\left(  T\right)  =\sum_{h\in\mathbb{N}_{+}}h\cdot\operatorname*{sig}%
\left(  \text{the }h\text{-th column of }T\right)  .
\]
\end{definition}

For instance, if $T$ is the 12-table
\begin{equation}
%
\ytableausetup{notabloids}
\begin{ytableau}
\none& \none& 1 & 2 & 1 & 2 \\
\none& 1 & 1 & 2 \\
2 & 1 & 1 & 2 \\
2 & 2
\end{ytableau}%
\label{eq.example-12-table.1}
\end{equation}
then $\ell\left(  T\right)  =1\cdot0+2\cdot1+3\cdot2+4\cdot0+5\cdot
2+6\cdot0+7\cdot0+8\cdot0+\cdots=18$.

\subsection{\label{subsection:benign}Descents, separators, and benign 12-tables}

In Subsection~\ref{sect.construction}, we have defined a ``descent''
of a 12-table. Let us reword this definition in more formal terms:
If $T$ is a 12-table, then we define a \textit{descent} of $T$ to be a
positive integer $i$ such that there exists an $r\in\mathbb{N}_{+}$ satisfying
$\left(  r,i\right)  \in \lm$, $\left(  r,i+1\right)  \in \lm$, $T\left(
r,i\right)  =2$ and $T\left(  r,i+1\right)  =1$. For instance,
the descents of the 12-table shown in (\ref{eq.example-12-table.1})
are $1$ and $4$. Clearly, a 12-rpp of shape $\lm$ is the same as a 12-table which has no
descents.

If $T$ is a 12-table, and if $k\in\mathbb{N}_{+}$ is such that the $k$-th
column of $T$ is mixed, then we define $\operatorname*{sep}\nolimits_{k}T$ to
be the smallest $r\in\mathbb{N}_{+}$ such that $\left(  r,k\right)  \in \lm$ and
$T\left(  r,k\right)  =2$. Thus, every 12-table $T$, every
$r\in\mathbb{N}_{+}$ and every $k\in\mathbb{N}_{+}$ such that the $k$-th
column of $T$ is mixed and such that $\left(  r,k\right)  \in \lm$ satisfy%
\begin{equation}
T\left(  r,k\right)  =\left\{
\begin{array}
[c]{l}%
1,\ \ \ \ \ \ \ \ \ \ \text{if }r<\operatorname*{sep}\nolimits_{k}T;\\
2,\ \ \ \ \ \ \ \ \ \ \text{if }r\geq\operatorname*{sep}\nolimits_{k}T.
\end{array}
\right.  \label{pf.lem.BK.Tsep}%
\end{equation}

If $T$ is a 12-table, then we let $\operatorname*{seplist}T$ denote the list
of all values $\operatorname*{sep}\nolimits_{k}T$ (in the order of increasing
$k$), where $k$ ranges over all positive integers for which the $k$-th column
of $T$ is mixed. For instance, if $T$ is
\[
%
\ytableausetup{notabloids}
\begin{ytableau}
\none& \none& 1 & 1 & 1 \\
\none& 2 & 1 & 1 & 2 \\
1 & 2 & 1 \\
2 & 2 & 2
\end{ytableau}%
\]
then $\operatorname*{sep}\nolimits_{1}T=4$, $\operatorname*{sep}\nolimits_{3}T=4$, and
$\operatorname*{sep}\nolimits_{5}T=2$ (and there are no other $k$ for which $\operatorname*{sep}\nolimits_{k}T$ is defined), so that $\operatorname*{seplist}T=\left(  4,4,2\right)  $.

We say that a 12-table $T$ is \textit{benign} if the list
$\operatorname*{seplist}T$ is weakly decreasing.\footnote{For
example, the 12-table in (\ref{eq.example-12-table.1}) is benign,
but replacing its third column by $\left(1,2,2\right)$ and its
fourth column by $\left(1,1,2\right)$ would yield a 12-table
which is not benign.}
Notice that 12-rpps are benign 12-tables, but the converse is not true. If $T$ is a benign 12-table, then%
\begin{align}
&  \text{there exists no descent }k\text{ of }T\text{ such that both the
}k\text{-th column of }T\nonumber\\
&  \text{and the }\left(  k+1\right)  \text{-th column of }T\text{ are mixed.}
\label{eq.benign.not-both-mixed}%
\end{align}
Let $\BenignTables$ denote the set of all benign 12-tables; we have $\OneTwoRPP\subseteq\BenignTables$.

Recall the map $\flip$ defined for 12-tables in Definition \ref{defi.flip}. If $T\in\BenignTables$ then $\flip(T)\in\BenignTables$ as well because $T$ and $\flip(T)$ have the same mixed columns. Thus, the map $\flip$ restricts to a map $\BenignTables\to\BenignTables$ which we will also denote $\flip$.

\begin{remark}
 \label{pf.lem.BK.flip.ircont}
It is clear that $\operatorname*{flip}$ is an involution on $\BenignTables$ that preserves $\operatorname*{ceq}$ and $\operatorname*{seplist}$ but switches the first two entries of $\operatorname*{ircont}$ (that is, if some $T \in \BenignTables$ has $\ircont\left(T\right) = \left(a,b,0,0,0,\ldots\right)$, then
$\operatorname{ircont}\left(\operatorname{flip}\left(T\right)\right) = \left(b,a,0,0,0,\ldots\right)$).
\end{remark}

\subsection{Plan of the proof}

Let us now briefly sketch the ideas behind the rest of the proof before we go
into them in detail. The map $\operatorname*{flip}:\BenignTables\rightarrow
\BenignTables$ does not generally send 12-rpps to 12-rpps (i.e., it does not
restrict to a map $\OneTwoRPP\rightarrow\OneTwoRPP$). However, we shall amend
this by defining a way to transform any benign 12-table into a 12-rpp by what
we call \textquotedblleft resolving descents\textquotedblright. The process of
\textquotedblleft resolving descents\textquotedblright\ will be a stepwise
process, and will be formalized in terms of a binary relation $\Rrightarrow$
on the set $\BenignTables$ which we will soon introduce. The intuition behind
saying \textquotedblleft$P\Rrightarrow Q$\textquotedblright\ is that the
benign 12-table $P$ has a descent, resolving which yields the benign 12-table
$Q$. Starting with a benign 12-table $P$, we can repeatedly resolve descents
until this is no longer possible. We have some freedom in
performing this process, because at any step there can be a choice of several
descents to resolve; but we will see that the final result does not depend on
the process. Hence, the final result can be regarded as a function of $P$. We
will denote it by $\operatorname*{norm}P$, and we will see that it is a
12-rpp. We will then define a map $\mathbf{B}:\OneTwoRPP\rightarrow\OneTwoRPP$
by $\mathbf{B}\left(  T\right)  =\operatorname*{norm}\left(
\operatorname*{flip}T\right)  $, and show that it is an involution satisfying
the properties that we want it to satisfy.

\subsection{\label{subsect.resolving}Resolving descents}

Now we come to the details.

Let $k\in\mathbb{N}_{+}$. Let $P\in\BenignTables$. Assume (for the whole Subsection \ref{subsect.resolving}) that $k$
is a descent of $P$. Thus, the $k$-th column of $P$ must contain at least one
$2$. Hence, the $k$-th column of $P$ is either mixed or 2-pure. Similarly, the
$\left(  k+1\right)  $-th column of $P$ is either mixed or 1-pure. But the
$k$-th and the $\left(  k+1\right)  $-th columns of $P$ cannot both be mixed
(by (\ref{eq.benign.not-both-mixed}), because $P$ is benign). Thus,
exactly one of the following three statements holds:

\begin{itemize}
\item[(M1)] The $k$-th column of $P$ is mixed and the $\left(  k+1\right)  $-th column of $P$ is 1-pure.

\item[(2M)] The $k$-th column of $P$ is 2-pure and the $\left(  k+1\right)  $-th column of $P$ is mixed.

\item[(21)] The $k$-th column of $P$ is 2-pure and the $\left(  k+1\right)  $-th column of $P$ is 1-pure.
\end{itemize}

Now, we define a new 12-table $\operatorname*{res}\nolimits_{k}P$ as follows (see Figure \ref{fig:des-res-preliminary} for illustration):

\begin{itemize}
\item If we have (M1), then $\operatorname*{res}_{k}P$ is the
12-table defined as follows:
The $k$-th column of $\operatorname*{res}_{k}P$ is 1-pure; the $\left(  k+1\right)  $-th column of $\operatorname*{res}%
\nolimits_{k}P$ is mixed and satisfies $\operatorname*{sep}\nolimits_{k+1}%
\left(  \operatorname*{res}\nolimits_{k}P\right)  =\operatorname*{sep}%
\nolimits_{k}P$; all other columns of $\operatorname*{res}\nolimits_{k}P$ are
copied over from $P$ unchanged.\footnote{The reader should check that
this 12-table is well-defined.
}

\item If we have (2M), then $\operatorname*{res}_{k}P$ is the
12-table defined as follows: The $k$-th column of $\operatorname*{res}_{k}P$
is mixed and satisfies $\operatorname*{sep}\nolimits_{k}\left(
\operatorname*{res}\nolimits_{k}P\right)  =\operatorname*{sep}\nolimits_{k+1}%
P$; the $\left(  k+1\right)  $-th column of $\operatorname*{res}%
\nolimits_{k}P$ is 2-pure; all other columns
of $\operatorname*{res}\nolimits_{k}P$ are copied over from $P$
unchanged.

\item If we have (21), then $\operatorname*{res}_{k}P$ is the
12-table defined as follows: The $k$-th column of $\operatorname*{res}_{k}P$
is 1-pure; the $\left(  k+1\right)  $-th column of $\operatorname*{res}%
\nolimits_{k}P$ is 2-pure; all other columns of $\operatorname*{res}%
\nolimits_{k}P$ are copied over from $P$ unchanged.
\end{itemize}

In either case, $\operatorname*{res}\nolimits_{k}P$ is a well-defined
12-table. It is furthermore clear that $\operatorname*{seplist}\left(
\operatorname*{res}\nolimits_{k}P\right)  =\operatorname*{seplist}P$. Thus,
$\operatorname*{res}\nolimits_{k}P$ is benign (since $P$ is benign); that is,
$\operatorname*{res}\nolimits_{k}P\in\BenignTables$. We say that
$\operatorname*{res}\nolimits_{k}P$ is the 12-table obtained by
\textit{resolving} the descent $k$ in $P$.

\Needspace{14\baselineskip}

\begin{example}
\label{exa.resolve.short}Let $P$ be the 12-table on the left:
\begin{center}
\begin{tabular}{||c||c||c||c||}\hline
 & & &  \\
 $\ytableausetup{notabloids}
\begin{ytableau}
\none& \none& 1 & 2 & 1 \\
\none& 1 & 1 & 2 \\
2 & 1 & 1 \\
2 & 2 & 1 \\
2
\end{ytableau}$ & 
 $\ytableausetup{notabloids}
\begin{ytableau}
\none& \none& 1 & 2 & 1 \\
\none& 2 & 1 & 2 \\
1 & 2 & 1 \\
2 & 2 & 1 \\
2
\end{ytableau}$ &
$\ytableausetup{notabloids}
\begin{ytableau}
\none& \none& 1 & 2 & 1 \\
\none& 1 & 1 & 2 \\
2 & 1 & 1 \\
2 & 1 & 2 \\
2
\end{ytableau}$
& 
$\ytableausetup{notabloids}
\begin{ytableau}
\none& \none& 1 & 1 & 2 \\
\none& 1 & 1 & 1 \\
2 & 1 & 1 \\
2 & 2 & 1 \\
2
\end{ytableau}$
\\
$P$ & $\operatorname*{res}\nolimits_{1}P$  & $\operatorname*{res}\nolimits_{2}P$  & $\operatorname*{res}\nolimits_{4}P$   \\\hline
\end{tabular}
\end{center}
Then $P$ is a benign 12-table, and its
descents are $1$, $2$ and $4$. We have $\operatorname*{sep}\nolimits_{2}P=4$.

If we set $k=1$ then we have (2M), if we set $k=2$ then we have (M1), and if we set $k=4$ then we have (21). We can resolve each of these three descents; the results are the three 12-tables on the right.

We notice that each of the three 12-tables $\operatorname*{res}\nolimits_{1}%
P$, $\operatorname*{res}\nolimits_{2}P$ and $\operatorname*{res}%
\nolimits_{4}P$ still has descents. In order to get a 12-rpp from
$P$, we will have to keep resolving these descents until none remain.
\end{example}


We now observe some further properties of $\operatorname*{res}\nolimits_{k}P$:

\begin{proposition}
\label{prop.descent-resolution-props}Let $P\in\BenignTables$ and $k\in
\mathbb{N}_{+}$ be such that $k$ is a descent of $P$.

\begin{enumerate}
\renewcommand{\theenumi}{\alph{enumi}}
\renewcommand{\labelenumi}{\textbf{(\theenumi)}}

\item \label{pf.lem.BK.res.loc}
The 12-table $\operatorname*{res}\nolimits_{k}P$ differs from $P$
only in columns $k$ and $k+1$.

\item \label{pf.lem.BK.res.loc2}
The $k$-th and the $\left(  k+1\right)  $-th columns of
$\operatorname*{res}\nolimits_{k}P$ depend only on the $k$-th and the $\left(
k+1\right)  $-th columns of $P$.

\item \label{pf.lem.BK.res.ceq}
We have%
\[
\operatorname*{ceq}\left(  \operatorname*{res}\nolimits_{k}P\right)
=\operatorname*{ceq}\left(  P\right) .
\]

\item \label{pf.lem.BK.res.irconts}
We have%
\[
\operatorname*{ircont}\left(  \operatorname*{res}\nolimits_{k}P\right)
=\operatorname*{ircont}\left(  P\right)  .
\]


\item \label{pf.lem.BK.res.flip}
The integer $k$ is a descent of $\operatorname*{flip}\left(  \operatorname*{res}%
\nolimits_{k}P\right)$, and we have
\[
\operatorname*{res}%
\nolimits_{k}\left(  \operatorname*{flip}\left(  \operatorname*{res}%
\nolimits_{k}P\right)  \right) = \operatorname*{flip}\left(  P\right).
\]

\item \label{pf.lem.BK.res.lendec}
Recall that we defined a nonnegative integer $\ell\left(
T\right)  $ for every 12-table $T$ in Definition \ref{defi.fourtypes}. We
have
\[
\ell\left(  P\right)  >\ell\left(  \operatorname*{res}\nolimits_{k}P\right)  .
\]

\end{enumerate}

\end{proposition}

\begin{proof}
[Proof of Proposition \ref{prop.descent-resolution-props}.]
All parts of Proposition \ref{prop.descent-resolution-props} follow from straightforward arguments
using the definitions of $\operatorname*{res}_{k}$ and $\operatorname*{flip}$
and (\ref{pf.lem.BK.Tsep}).
\end{proof}

\subsection{The descent-resolution relation $\Rrightarrow$}

\begin{definition}
Let us now define a binary relation $\Rrightarrow$ on the set $\BenignTables$ as
follows: Let $P\in\BenignTables$ and $Q\in\BenignTables$. If $k\in\mathbb{N}_{+}$,
then we write $P\underset{k}{\Rrightarrow}Q$ if $k$ is a descent
of $P$ and we have $Q=\operatorname*{res}\nolimits_{k}P$. We write $P\Rrightarrow Q$ if there
exists a $k\in\mathbb{N}_{+}$ such that $P\underset{k}{\Rrightarrow}Q$.
\end{definition}

Proposition \ref{prop.descent-resolution-props} translates into the following properties of this relation
$\Rrightarrow$:

\begin{lemma}
\label{lem.descent-resolution-props}Let $P\in\BenignTables$ and $Q\in\BenignTables$
be such that $P\Rrightarrow Q$. Then:

\begin{enumerate}
\item[\textbf{(a)}] We have $\operatorname*{ceq}\left(  Q\right)
=\operatorname*{ceq}\left(  P\right)  $.

\item[\textbf{(b)}] We have $\operatorname*{ircont}\left(  Q\right)
=\operatorname*{ircont}\left(  P\right)  $.

\item[\textbf{(c)}] The benign 12-tables $\operatorname*{flip}\left(  P\right)  $ and
$\operatorname*{flip}\left(  Q\right)  $ have the property that
$\operatorname*{flip}\left(  Q\right)  \Rrightarrow\operatorname*{flip}\left(
P\right)  $.

\item[\textbf{(d)}] We have $\ell\left(  P\right)  >\ell\left(  Q\right)  $.

\end{enumerate}
\end{lemma}

We now define $\overset{\ast}{\Rrightarrow}$ to be the
reflexive-and-transitive closure of the relation $\Rrightarrow$.
\footnote{Explicitly, this means that $\overset{\ast}{\Rrightarrow}$
is defined as follows: For two elements $P\in\BenignTables$ and $Q\in\BenignTables$,
we have $P\overset{\ast}{\Rrightarrow}Q$ if and only if there exists a
sequence $\left(  a_{0},a_{1},\ldots,a_{n}\right)  $ of elements of
$\BenignTables$ such that $a_{0}=P$ and $a_{n}=Q$ and such that every
$i\in\left\{  0,1,\ldots,n-1\right\}  $ satisfies $a_{i}\Rrightarrow a_{i+1}%
$. (Note that $n$ is allowed to be $0$.)}
This relation $\overset{\ast}{\Rrightarrow}$ is reflexive and transitive,
and extends the relation $\Rrightarrow$.
Lemma \ref{lem.descent-resolution-props} thus yields:

\begin{lemma}
\label{lem.descent-resolution-*props}Let $P\in\BenignTables$ and $Q\in\BenignTables$
be such that $P\overset{\ast}{\Rrightarrow}Q$. Then:

\begin{enumerate}

\item[\textbf{(a)}] We have $\operatorname*{ceq}\left(  Q\right)
=\operatorname*{ceq}\left(  P\right)  $.

\item[\textbf{(b)}] We have $\operatorname*{ircont}\left(  Q\right)
=\operatorname*{ircont}\left(  P\right)  $.

\item[\textbf{(c)}] The benign 12-tables $\operatorname*{flip}\left(  P\right)  $ and
$\operatorname*{flip}\left(  Q\right)  $ have the property that
$\operatorname*{flip}\left(  Q\right)  \overset{\ast}{\Rrightarrow
}\operatorname*{flip}\left(  P\right)  $.

\item[\textbf{(d)}] We have $\ell\left(  P\right)  \geq\ell\left(  Q\right)  $.

\end{enumerate}
\end{lemma}


We now state the following crucial lemma:

\begin{lemma}
\label{prop.descent-resolution-hyps}Let $A$, $B$ and $C$ be three elements of
$\BenignTables$ satisfying $A\Rrightarrow B$ and $A\Rrightarrow C$. Then, there
exists a $D\in\BenignTables$ such that $B\overset{\ast}{\Rrightarrow}D$ and
$C\overset{\ast}{\Rrightarrow}D$.
\end{lemma}

\begin{proof}
[Proof of Lemma \ref{prop.descent-resolution-hyps}.]If $B=C$, then we can
simply choose $D=B=C$; thus, we assume that $B\neq C$.

Let $u$, $v \in\mathbb{N}_{+}$ be such that $A\underset{u}{\Rrightarrow}B$ and $A\underset{v}{\Rrightarrow}C$.
Thus, $B = \operatorname{res}_u A$ and $C = \operatorname{res}_v A$.
Since $B \neq C$, we have $u \neq v$.
Without loss of generality, assume that $u<v$. We are in one of the following two cases:

\textit{Case 1:} We have $u=v-1$.

\textit{Case 2:} We have $u<v-1$.

Let us deal with Case 2 first. In this
case, $\left\{  u,u+1\right\}  \cap\left\{  v,v+1\right\}
=\varnothing$.
It follows that
$\operatorname*{res}\nolimits_{v}\left(  \operatorname*{res}\nolimits_{u}%
A\right)  $ and $\operatorname*{res}\nolimits_{u}%
\left(  \operatorname*{res}\nolimits_{v}A\right)  $ are well-defined and $\operatorname*{res}\nolimits_{u}\left(
\operatorname*{res}\nolimits_{v}A\right)  =\operatorname*{res}\nolimits_{v}%
\left(  \operatorname*{res}\nolimits_{u}A\right) $. Setting
$D=\operatorname*{res}\nolimits_{u}\left(  \operatorname*{res}\nolimits_{v}%
A\right)  =\operatorname*{res}\nolimits_{v}\left(  \operatorname*{res}%
\nolimits_{u}A\right)  $ completes the proof in this case.

Now, let us consider Case 1.
The $v$-th column of $A$ must contain a $1$ (since $v-1 = u$ is a descent of $A$)
and a $2$ (since $v$ is a descent of $A$). Hence, the $v$-th column of $A$ is
mixed. Since $A$ is benign but has $v-1$ and $v$ as descents, it thus follows
that the $\left(  v-1\right)  $-th column of $A$ is
2-pure and the
$\left(  v+1\right)  $-th column of $A$ is 1-pure. We can represent the
relevant portion (that is, the $\left(v-1\right)$-th, $v$-th and
$\left(v+1\right)$-th columns) of the 12-table $A$ as follows:
\begin{equation}
A=%
\begin{tabular}{ccc}
\cline{3-3} & & \multicolumn{1}{|c|}{} \\
\cline{2-2} & \multicolumn{1}{|c|}{} & \multicolumn{1}{|c|}{} \\
\cline{1-1} \multicolumn{1}{|c|}{} & \multicolumn{1}{|c|}{1} & \multicolumn
{1}{|c|}{1} \\
\multicolumn{1}{|c|}{} & \multicolumn{1}{|c|}{} & \multicolumn{1}{|c|}{} \\
\cline{2-2} \multicolumn{1}{|c|}{2} & \multicolumn{1}{|c|}{} & \multicolumn
{1}{|c|}{} \\
\multicolumn{1}{|c|}{} & \multicolumn{1}{|c|}{2} & \multicolumn{1}{|c|}{} \\
\cline{3-3} \multicolumn{1}{|c|}{} & \multicolumn{1}{|c|}{} \\
\cline{2-2} \multicolumn{1}{|c|}{} \\
\cline{1-1}
\end{tabular}%
. \label{pf.prop.descent-resolution.hyps.short.A}%
\end{equation}
Notice that the separating line which separates the $1$'s from the
$2$'s in column $v$ is lower than the upper border of the
$\left(v-1\right)$-th column (since $v-1$ is a descent of $A$), and
higher than the lower border of the $\left(v+1\right)$-th column
(since $v$ is a descent of $A$).

%

Let $s=\operatorname*{sep}\nolimits_{v}A$.
Then, the cells
$\left(s, v-1\right)$, $\left(s, v\right)$, $\left(s, v+1\right)$,
$\left(s+1, v-1\right)$, $\left(s+1, v\right)$, $\left(s+1, v+1\right)$
all belong to $\lm$ (due to what we just said about separating
lines). We shall refer to this observation as the
``six-cells property''.

Now, $B = \operatorname{res}_u A = \operatorname{res}_{v-1}A$, so
$B$ is represented as follows:
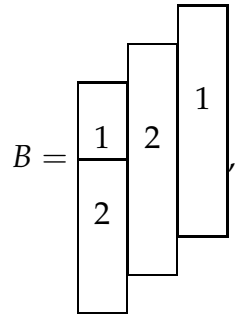
\[
B=%
\begin{tabular}{ccc}
\cline{3-3} & & \multicolumn{1}{|c|}{} \\
\cline{2-2} & \multicolumn{1}{|c|}{} & \multicolumn{1}{|c|}{} \\
\cline{1-1} \multicolumn{1}{|c|}{} & \multicolumn{1}{|c|}{} & \multicolumn
{1}{|c|}{1} \\
\multicolumn{1}{|c|}{1} & \multicolumn{1}{|c|}{2} & \multicolumn{1}{|c|}{} \\
\cline{1-1} \multicolumn{1}{|c|}{} & \multicolumn{1}{|c|}{} & \multicolumn
{1}{|c|}{} \\
\multicolumn{1}{|c|}{2} & \multicolumn{1}{|c|}{} & \multicolumn{1}{|c|}{} \\
\cline{3-3} \multicolumn{1}{|c|}{} & \multicolumn{1}{|c|}{} \\
\cline{2-2} \multicolumn{1}{|c|}{} \\
\cline{1-1}
\end{tabular}%
,
\]
where $\operatorname*{sep}\nolimits_{v-1}B = s$
(that is, the separating line in the
$\left(v-1\right)$-th column of $B$
is between the cells $\left(s,v-1\right)$ and
$\left(s+1,v-1\right)$).
Now, $v$ is a descent of $B$.
Resolving this descent yields a
12-table $\operatorname{res}_{v}B$ which is represented as follows:
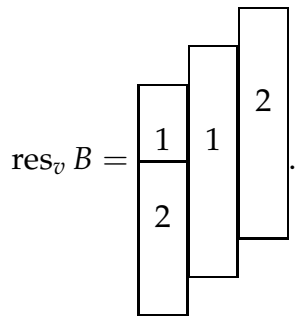
\[
\operatorname*{res}\nolimits_{v}B=%
\begin{tabular}{ccc}
\cline{3-3} & & \multicolumn{1}{|c|}{} \\
\cline{2-2} & \multicolumn{1}{|c|}{} & \multicolumn{1}{|c|}{} \\
\cline{1-1} \multicolumn{1}{|c|}{} & \multicolumn{1}{|c|}{} & \multicolumn
{1}{|c|}{2} \\
\multicolumn{1}{|c|}{1} & \multicolumn{1}{|c|}{1} & \multicolumn{1}{|c|}{} \\
\cline{1-1} \multicolumn{1}{|c|}{} & \multicolumn{1}{|c|}{} & \multicolumn
{1}{|c|}{} \\
\multicolumn{1}{|c|}{2} & \multicolumn{1}{|c|}{} & \multicolumn{1}{|c|}{} \\
\cline{3-3} \multicolumn{1}{|c|}{} & \multicolumn{1}{|c|}{} \\
\cline{2-2} \multicolumn{1}{|c|}{} \\
\cline{1-1}
\end{tabular}%
.
\]
This, in turn, shows that $v-1$ is a descent of $\operatorname*{res}%
\nolimits_{v}B$ (by the six-cells property). Resolving this descent
yields a 12-table $\operatorname*{res}\nolimits_{v-1}\left(
\operatorname*{res}\nolimits_{v}B\right)  $ which is represented as follows:%
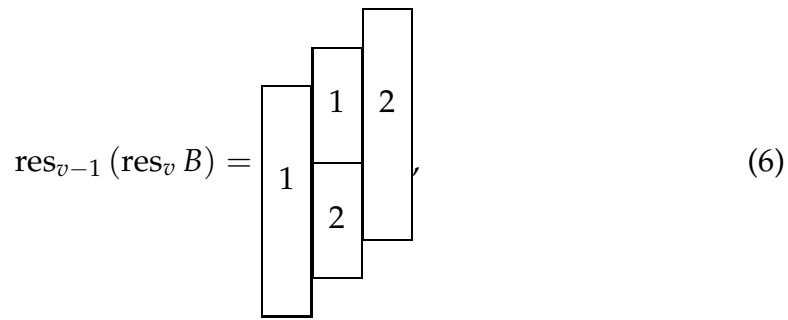
\begin{equation}
\operatorname*{res}\nolimits_{v-1}\left(  \operatorname*{res}\nolimits_{v}%
B\right)  =%
\begin{tabular}{ccc}
\cline{3-3} & & \multicolumn{1}{|c|}{} \\
\cline{2-2} & \multicolumn{1}{|c|}{} & \multicolumn{1}{|c|}{} \\
\cline{1-1} \multicolumn{1}{|c|}{} & \multicolumn{1}{|c|}{1} & \multicolumn
{1}{|c|}{2} \\
\multicolumn{1}{|c|}{} & \multicolumn{1}{|c|}{} & \multicolumn{1}{|c|}{} \\
\cline{2-2} \multicolumn{1}{|c|}{1} & \multicolumn{1}{|c|}{} & \multicolumn
{1}{|c|}{} \\
\multicolumn{1}{|c|}{} & \multicolumn{1}{|c|}{2} & \multicolumn{1}{|c|}{} \\
\cline{3-3} \multicolumn{1}{|c|}{} & \multicolumn{1}{|c|}{} \\
\cline{2-2} \multicolumn{1}{|c|}{} \\
\cline{1-1}
\end{tabular}%
, \label{pf.prop.descent-resolution.hyps.short.D}%
\end{equation}
where $\operatorname*{sep}\nolimits_{v} \left( \operatorname*{res}\nolimits_{v-1}\left(  \operatorname*{res}\nolimits_{v}%
B\right) \right) = s$.

On the other hand, $C = \operatorname*{res}\nolimits_{v}A$. We can apply a similar argument as above to show that the 12-table $\operatorname*{res}\nolimits_{v}\left(
\operatorname*{res}\nolimits_{v-1}C\right)  $ is well-defined, and is exactly equal to the 12-table in (\ref{pf.prop.descent-resolution.hyps.short.D}).
Hence, $\operatorname*{res}%
\nolimits_{v-1}\left(  \operatorname*{res}\nolimits_{v}B\right) = \operatorname*{res}\nolimits_{v}\left(
\operatorname*{res}\nolimits_{v-1}C\right)$, and setting $D$ equal to this 12-table completes the proof in Case 1.
\end{proof}

\subsection{The normalization map}

The following proposition is the most important piece in our puzzle:

\begin{proposition}
\label{prop.BK.norm}For every $T\in\BenignTables$, there exists a unique
$N\in\OneTwoRPP$ such that $T\overset{\ast}{\Rrightarrow}N$.
\end{proposition}

\begin{proof}
[Proof of Proposition \ref{prop.BK.norm}.] For every $T\in\BenignTables$, let
$\operatorname*{Norm}\left(  T\right)  $ denote the set%
\[
\left\{  N\in\OneTwoRPP\ \mid\ T\overset{\ast}{\Rrightarrow}N\right\}  .
\]
Thus, in order to prove Proposition \ref{prop.BK.norm}, we need to show that for every $T\in\BenignTables$ this set $\operatorname*{Norm}%
\left(  T\right)  $ is a one-element set.

We shall prove this by strong induction on $\ell\left(  T\right)  $. Fix
some $T\in\BenignTables$, and assume that
\begin{equation}
\operatorname*{Norm}\left(  S\right)  \text{ is a one-element set for every
}S\in\BenignTables\text{ satisfying }\ell\left(  S\right)  <\ell\left(  T\right)
\text{.}\label{pf.prop.BK.norm.indhyp}%
\end{equation}
We then need to prove that $\operatorname*{Norm}\left(  T\right)  $ is a
one-element set.

Let $\mathbf{Z} = \left\{  S\in\BenignTables\ \mid\ T\Rrightarrow
S\right\}  $. In other words, $\mathbf{Z}$ is the set of all benign 12-tables
$S$ which can be obtained from $T$ by resolving one descent. If
$\mathbf{Z}$ is empty, then $T \in \OneTwoRPP$, so that
$\operatorname*{Norm}\left(  T\right) = \{T\}$ and we are done. Hence, we can
assume that $\mathbf{Z}$ is nonempty. Therefore $T \notin \OneTwoRPP$.

Thus, every $N \in \OneTwoRPP$ satisfying $T \overset{\ast}{\Rrightarrow} N$
must satisfy $Z \overset{\ast}{\Rrightarrow} N$ for some $Z \in \mathbf{Z}$.
In other words, every $N \in \operatorname{Norm}\left(T\right)$ must belong
to $\operatorname{Norm}\left(Z\right)$ for some $Z \in \mathbf{Z}$. The
converse of this clearly holds as well. Hence, 
\begin{equation}
\operatorname*{Norm}\left(  T\right)  =\bigcup_{Z\in\mathbf{Z}%
}\operatorname*{Norm}\left(  Z\right)  .
\label{pf.prop.BK.norm.union}
\end{equation}

Let us now notice that:

\begin{itemize}
\item By Lemma \ref{lem.descent-resolution-props} \textbf{(d)} and
(\ref{pf.prop.BK.norm.indhyp}), for every $Z\in\mathbf{Z}$, the set
$\operatorname*{Norm}\left( Z\right)  $ is a one-element set.

\item By Lemma \ref{prop.descent-resolution-hyps}, for every
$B\in\mathbf{Z}$ and $C\in\mathbf{Z}$, we have
$\operatorname*{Norm}\left(  B\right)  \cap\operatorname*{Norm}\left(
C\right)  \neq\varnothing$. 
\ \ \ \ \footnote{In more detail:
Let $B\in\mathbf{Z}$ and $C\in\mathbf{Z}$.
By Lemma \ref{prop.descent-resolution-hyps} (applied to $A=T$)
there exists
a $D\in\BenignTables$ such that $B\overset{\ast}{\Rrightarrow}D$ and
$C\overset{\ast}{\Rrightarrow}D$. This $D$ has
$\ell\left( T\right)  >\ell\left(  B\right)
\geq\ell\left(  D\right)  $ (by Lemma \ref{lem.descent-resolution-props}
\textbf{(d)} and
Lemma \ref{lem.descent-resolution-*props} \textbf{(d)}, respectively).
Hence, by (\ref{pf.prop.BK.norm.indhyp}), the set
$\operatorname*{Norm}\left(  D\right)  $ is a one-element set. Its unique element
clearly lies in both $\operatorname*{Norm}\left(  B\right)  $ and
$\operatorname*{Norm}\left(  C\right)  $, so
$\operatorname*{Norm}\left(
B\right)  \cap\operatorname*{Norm}\left(  C\right)  \neq\varnothing$.

}

\end{itemize}

Hence, (\ref{pf.prop.BK.norm.union}) shows that $\operatorname*{Norm}\left(
T\right)  $ is a union of one-element sets, any two of which have a nonempty
intersection (and thus are identical).
Moreover, this union is nonempty (since $\mathbf{Z}$ is nonempty). Hence,
$\operatorname*{Norm}\left(  T\right)  $ itself is a one-element set. This
completes our induction.
\end{proof}

\begin{definition}
Let $T\in\BenignTables$. Proposition \ref{prop.BK.norm} shows that there exists a
unique $N\in\OneTwoRPP$ such that $T\overset{\ast}{\Rrightarrow}N$. We define
$\operatorname*{norm}\left(  T\right)  $ to be this $N$.
\end{definition}

\subsection{Definition of $\mathbf{B}$}

\begin{definition}
\label{defi.B}
Let us define a map $\mathbf{B}:\OneTwoRPP\rightarrow\OneTwoRPP$ as follows:
For every $T\in\OneTwoRPP$, set $\mathbf{B}\left(
T\right) = \operatorname*{norm}\left(  \operatorname*{flip}\left(
T\right)  \right)  $.
\end{definition}

In order to complete the proof of Lemma
\ref{lem.BK}, we need to show that $\mathbf{B}$ is an involution, preserves the ceq statistic, and switches the number of columns containing a $1$ with the number of columns containing a $2$. At this point, all of this is easy:

\begin{proof}[$\mathbf{B}$ is an involution]
Let $T\in \OneTwoRPP$. We have $\operatorname*{flip}\left(
T\right)  \overset{\ast}{\Rrightarrow}\operatorname*{norm}\left(
\operatorname*{flip}\left(  T\right)  \right)  =
\mathbf{B}\left(  T\right)  $. Lemma \ref{lem.descent-resolution-*props} \textbf{(c)} thus yields $\operatorname*{flip}\left(  \mathbf{B}\left(  T\right)
\right)  \overset{\ast}{\Rrightarrow}\operatorname*{flip}\left(
\operatorname*{flip}T\right)  = T$.

But $\mathbf{B}(\mathbf{B}(T)) = \operatorname*{norm}\left(  \operatorname*{flip}\left(
\mathbf{B}\left(  T\right)  \right)  \right)  $ is the unique $N\in\OneTwoRPP$
such that $\operatorname*{flip}\left(  \mathbf{B}\left(  T\right)  \right)
\overset{\ast}{\Rrightarrow}N$. Since $T \in \OneTwoRPP$, we have $\mathbf{B}(\mathbf{B}(T)) = T$, as desired.
\end{proof}

\begin{proof}[$\mathbf{B}$ preserves $\ceq$]
Let $T\in\OneTwoRPP$. As above, $\operatorname*{flip}\left(  T\right)  \overset{\ast
}{\Rrightarrow}\mathbf{B}\left(  T\right)  $. Lemma \ref{lem.descent-resolution-*props} \textbf{(a)} and Remark \ref{pf.lem.BK.flip.ircont} thus yield
$\operatorname*{ceq}\left(  \mathbf{B}\left(  T\right)  \right)
=\operatorname*{ceq}\left(  \operatorname*{flip}\left(  T\right)  \right)
=\operatorname*{ceq}\left(  T\right)$.
\end{proof}

\begin{proof}[$\mathbf{B}$ switches the numbers of columns containing 1 and 2]
Let $T\in\OneTwoRPP$. As above, $\operatorname*{flip}\left(  T\right)  \overset{\ast
}{\Rrightarrow}\mathbf{B}\left(  T\right)  $. Lemma \ref{lem.descent-resolution-*props} \textbf{(b)}
thus yields
$\operatorname*{ircont}\left(  \mathbf{B}\left(  T\right)  \right)
=\operatorname*{ircont}\left(  \operatorname*{flip}\left(  T\right)  \right)$.
Due to Remark \ref{pf.lem.BK.flip.ircont}, this is the result
of switching the first two entries of $\ircont\left(T\right)$.
\end{proof}

Lemma \ref{lem.BK} is now proven.

\section{\label{sect.BKclassical}The classical Bender-Knuth involutions}

\def\BK{{\mathbf{BK}}}

Fix a skew partition $\lambda/\mu$ and a positive integer $i$.
We claim that the involution $\mathbf{B}_{i}:\operatorname*{RPP}\left(
\lambda/\mu\right)  \rightarrow\operatorname*{RPP}\left(  \lambda/\mu\right)
$ we have constructed in the proof of Theorem~\ref{thm.BK}
is a generalization of the $i$-th Bender-Knuth involution defined for
semistandard tableaux. First, we shall
define the $i$-th Bender-Knuth involution (following \cite[proof of
Proposition 2.11]{GriRei15} and \cite[proof of Theorem 7.10.2]{Stan99}).

Let $\operatorname*{SST}\left(  \lambda/\mu\right)  $ denote the set of all
semistandard tableaux of shape $ \lm$. We define a
map $\BK_{i}:\operatorname*{SST}\left(  \lambda/\mu\right)  \rightarrow
\operatorname*{SST}\left(  \lambda/\mu\right)  $ as follows:

Let $T\in\operatorname*{SST}\left(  \lambda/\mu\right)  $. Then every column of $T$ contains at most one $i$ and
at most one $i+1$. If a column contains both an
$i$ and an $i+1$, we will mark its entries as
\textquotedblleft ignored\textquotedblright. Now, let $k\in\mathbb{N}_{+}$.
The $k$-th row of $T$ is a weakly increasing sequence of positive integers;
thus, it contains a (possibly empty) string of $i$'s followed by a (possibly
empty) string of $\left(  i+1\right)  $'s. These two strings together form a
substring of the $k$-th row which looks as follows:%
\[
\left(  i,i,\ldots,i,i+1,i+1,\ldots,i+1\right).
\]
Some of the entries of this substring are \textquotedblleft
ignored\textquotedblright; it is easy to see that the \textquotedblleft
ignored\textquotedblright\ $i$'s are gathered at the left end of the substring
whereas the \textquotedblleft ignored\textquotedblright\ $\left(  i+1\right)
$'s are gathered at the right end of the substring. So the substring looks
as follows:
\[
\left(  \underbrace{i,i,\ldots,i}_{\substack{a\text{ many }i\text{'s
which}\\\text{are \textquotedblleft ignored\textquotedblright}}%
},\underbrace{i,i,\ldots,i}_{\substack{r\text{ many }i\text{'s which}%
\\\text{are not \textquotedblleft ignored\textquotedblright}}%
},\underbrace{i+1,i+1,\ldots,i+1}_{\substack{s\text{ many }\left(  i+1\right)
\text{'s which}\\\text{are not \textquotedblleft ignored\textquotedblright}%
}},\underbrace{i+1,i+1,\ldots,i+1}_{\substack{b\text{ many }\left(
i+1\right)  \text{'s which}\\\text{are \textquotedblleft
ignored\textquotedblright}}}\right)
\]
for some $a,r,s,b\in\mathbb{N}$. Now, we change this substring into%
\[
\left(  \underbrace{i,i,\ldots,i}_{\substack{a\text{ many }i\text{'s
which}\\\text{are \textquotedblleft ignored\textquotedblright}}%
},\underbrace{i,i,\ldots,i}_{\substack{s\text{ many }i\text{'s which}%
\\\text{are not \textquotedblleft ignored\textquotedblright}}%
},\underbrace{i+1,i+1,\ldots,i+1}_{\substack{r\text{ many }\left(  i+1\right)
\text{'s which}\\\text{are not \textquotedblleft ignored\textquotedblright}%
}},\underbrace{i+1,i+1,\ldots,i+1}_{\substack{b\text{ many }\left(
i+1\right)  \text{'s which}\\\text{are \textquotedblleft
ignored\textquotedblright}}}\right)  .
\]
We do this for every $k\in\mathbb{N}_{+}$. At the end, we have obtained a new semistandard
tableau of shape $\lambda/\mu$. We define $\BK_{i}\left(
T\right)  $ to be this new tableau.

\begin{proposition}
\label{prop.BKclassical}The map $\BK_{i}:\operatorname*{SST}\left(  \lambda
/\mu\right)  \rightarrow\operatorname*{SST}\left(  \lambda/\mu\right)  $ thus
defined is an involution. It is known as the $i$\textit{-th Bender-Knuth
involution}.
\end{proposition}

Now, every semistandard tableau of shape $\lm$ is
also an rpp of shape $\lm$. Hence, $\mathbf{B}_{i}\left(  T\right)  $ is
defined for every $T\in\operatorname*{SST}\left(  \lambda/\mu\right)  $. Our claim is the following:

\begin{proposition}
\label{prop.BK=BK}For every $T\in\operatorname*{SST}\left(  \lambda
/\mu\right)  $, we have $\BK_{i}\left(  T\right)  =\mathbf{B}_{i}\left(
T\right)  $.
\end{proposition}

\begin{proof}
[Proof of Proposition \ref{prop.BK=BK}.] 
Recall that the map $\mathbf{B}_i$ comes from the map $\mathbf{B}$ we defined on 12-rpps in Section~\ref{sect.proof}.
We could have constructed the map $\BK_i$ from the map $\BK_1$ in an analogous way.
We define a \textit{12-sst} to be a semistandard tableau whose entries all belong to the set $\left\{  1,2\right\}  $.
Clearly, to prove Proposition~\ref{prop.BK=BK}, it suffices to prove that $\BK_1(T) = \mathbf{B}(T)$ for all 12-ssts $T$.

Let $T$ be a 12-sst, and let $k \in \mathbb{N}_+$. The $k$-th row of $T$ has the form
\[
\left(  \underbrace{1,1,\ldots,1}_{\substack{a \ 1\text{'s
which are in}\\\text{mixed columns}}%
},\underbrace{1,1,\ldots,1}_{\substack{r \ 1\text{-pure}%
\\\text{columns}}%
},\underbrace{2,2,\ldots,2}_{\substack{s \ 2\text{-pure}%
\\\text{columns}}%
},\underbrace{2,2,\ldots,2}_{\substack{b \ 2\text{'s
which are in}\\\text{mixed columns}}}\right)
\]
where we use the observation that each 1-pure and each 2-pure column contains only one entry. Thus, the $k$-th row of $\operatorname*{flip}\left(  T\right)$ is
\[
\left(  \underbrace{1,1,\ldots,1}_{\substack{a \ 1\text{'s
which are in}\\\text{mixed columns}}%
},\underbrace{2,2,\ldots,2}_{\substack{r \ 2\text{-pure}%
\\\text{columns}}%
},\underbrace{1,1,\ldots,1}_{\substack{s \ 1\text{-pure}%
\\\text{columns}}%
},\underbrace{2,2,\ldots,2}_{\substack{b \ 2\text{'s
which are in}\\\text{mixed columns}}}\right).
\]
We can now repeatedly apply descent-resolution steps to obtain a tableau whose $k$-th row is
\[
\left(  \underbrace{1,1,\ldots,1}_{\substack{a \ 1\text{'s
which are in}\\\text{mixed columns}}%
},\underbrace{1,1,\ldots,1}_{\substack{s \ 1\text{-pure}%
\\\text{columns}}%
},\underbrace{2,2,\ldots,2}_{\substack{r \ 2\text{-pure}%
\\\text{columns}}%
},\underbrace{2,2,\ldots,2}_{\substack{b \ 2\text{'s
which are in}\\\text{mixed columns}}}\right).
\]
Repeating this process for every row of $\operatorname*{flip}\left(  T\right)$ (we can do this because each pure column contains only one entry, and thus each descent-resolution described above affects only one row), we obtain a 12-rpp. By the definition of $\mathbf{B}$, this rpp must equal $\mathbf{B}(T)$. By the above description, it is also equal to $\BK_1(T)$ (because the ignored entries in the construction of $\BK_1(T)$ are precisely the entries lying in mixed columns), which completes the proof.
\end{proof}

\section{The structure of 12-rpps}
\label{sect.structure}

In this section, 
we restrict ourselves to the
two-variable dual stable Grothendieck polynomial
$\g_\lm(x_1,x_2,0,0,\dots;\t)$ defined as the result of
substituting $0, 0, 0, \ldots$ for $x_3, x_4, x_5, \ldots$
in $\g_\lm$. We can represent it as a polynomial in
$\t$ with coefficients in $\Z[x_1,x_2]$:
$$\g_\lm(x_1,x_2,0,0,\dots;\t)=\sum_{\ceqvar\in \N^{\N_+}}\t^\ceqvar Q_\ceqvar(x_1,x_2),$$
where the sum ranges over all weak compositions $\alpha$,
and all but finitely many $Q_\ceqvar(x_1,x_2)$ are $0$.

We shall show that each $Q_\ceqvar(x_1,x_2)$ is either zero or has the form
\begin{equation}
\label{eq.Qalpha}
Q_\ceqvar(x_1,x_2)=(x_1x_2)^{M} P_{n_0}(x_1,x_2)P_{n_1}(x_1,x_2)\cdots P_{n_r}(x_1,x_2) , 
\end{equation}
where $M,r$ and $n_0,n_1,\dots,n_{r}$ are nonnegative integers naturally associated to $\ceqvar$ and $\lm$ and 
$$P_n(x_1,x_2)=\frac{x_1^{n+1}-x_2^{n+1}}{x_1-x_2}=x_1^{n}+x_1^{n-1}x_2+\cdots+x_1x_2^{n-1}+x_2^n.$$
We fix the skew partition $\lm$ throughout the whole section.
We will have a running example with $\lambda=(7,7,7,4,4)$ and $\mu=(5,3,2)$.

\subsection{Irreducible components}
We recall that a \textit{12-rpp} means an rpp whose entries all belong to the set $\left\{1, 2\right\}$. 

Given a 12-rpp $T$, consider the set $\NS(T)$ of all cells $(i,j)\in \lm$ such that $T(i,j)=1$ but $(i+1,j) \in \lm$ and $T(i+1,j)=2$. (In other words, $\NS(T)$ is the set of all non-redundant cells in $T$ which are filled with a $1$ and which are not the lowest cells in their columns.)
Clearly, $\NS(T)$ contains at most one cell from each column; thus, let us write $\NS(T)=\{(i_1,j_1),(i_2,j_2),\dots,(i_s,j_s)\}$ with $j_1<j_2<\cdots<j_s$. Because $T$ is a 12-rpp, it follows that the numbers $i_1,i_2,\dots,i_s$ decrease weakly, therefore they form a partition which we
denoted 
$$\seplist(T):=(i_1,i_2,\dots,i_s)$$ 
in Section \ref{subsection:benign}. This partition
will be called the \textit{seplist-partition of $T$}.
An example of calculation of $\seplist(T)$ and $\NS(T)$ is illustrated on Figure \ref{fig:seplist}. 

\def\one{{\mathbf{1}}}
\def\two{{\mathbf{2}}}

\begin{figure}
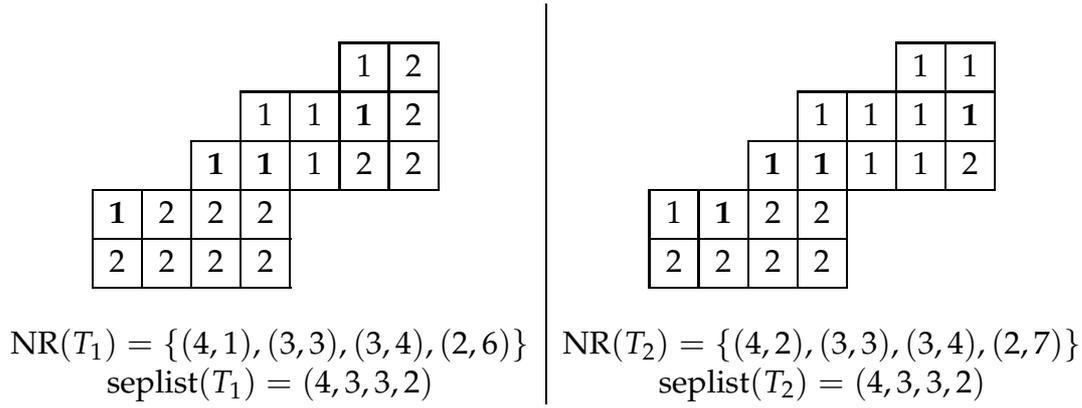

\begin{tabular}{c|c}
 & \\
\begin{ytableau}
\none& \none& \none&\none&\none & 1   & 2 \\
\none& \none& \none& 1   & 1    &\one & 2 \\
\none& \none& \one &\one & 1    & 2   & 2 \\
\one &    2 & 2    & 2 \\
2    &    2 & 2    & 2
\end{ytableau}\ \  &
\begin{ytableau}
\none& \none& \none&\none&\none & 1   & 1    \\
\none& \none& \none& 1   & 1    & 1   & \one \\
\none& \none& \one &\one & 1    & 1   & 2    \\
1    & \one & 2    & 2 \\
2    &    2 & 2    & 2
\end{ytableau}\\
 & \\
$\NS(T_1)=\{(4,1),(3,3),(3,4),(2,6)\}$ & $\NS(T_2)=\{(4,2),(3,3),(3,4),(2,7)\}$ \\
$\seplist(T_1)=(4,3,3,2)$ & $\seplist(T_2)=(4,3,3,2)$ 
\end{tabular}\\
\caption{\label{fig:seplist} Two 12-rpps of the same shape and with the same seplist-partition.}
\end{figure}

We would like to answer the following question: for which partitions $\seplistvar=(i_1\geq \cdots\geq i_s>0)$ does there exist a 12-rpp $T$ of shape $\lm$ such that $\seplist(T)=\seplistvar$?

A trivial necessary condition for this to happen is that there should exist some numbers $j_1<j_2<\cdots<j_s$ such that 
\begin{equation}\label{cond:necessary}
(i_1,j_1),(i_1+1,j_1),(i_2,j_2),(i_2+1,j_2),\dots,(i_s,j_s),(i_s+1,j_s)\in \lm. 
\end{equation}

Until the end of Section \ref{sect.structure}, we make an
assumption: namely, that the skew partition $\lm$ is
connected as a subgraph of $\Z^2$ (where two nodes are connected
if and only if their cells have an edge in common), and that
it has no empty columns. This is a harmless assumption,
since every skew partition $\lm$ can be written as a disjoint union
of such connected skew partitions and the corresponding seplist-partition splits into several independent parts, the polynomials $\g_\lm$ get multiplied and the right hand side of (\ref{eq.Qalpha}) changes accordingly. 

\newcommand{\nuxy}[2]{\seplistvar\big|_{\subseteq[#1,#2)}}
\newcommand{\nuxycap}[2]{\seplistvar\big|_{\cap[#1,#2)}}
\newcommand{\nupxy}[2]{\seplistvar^{\prime}\big|_{\subseteq[#1,#2)}}

\def\nuab{\nuxy{a}{b}}
\def\nuabcap{\nuxycap{a}{b}}

For each integer $i$, the set of all integers $j$ such that $(i,j),(i+1,j)\in\lm$ is just an interval $[\mu_{i}+1,\lambda_{i+1}]$, which we call \textit{the support of $i$} and denote $\supp(i):=[\mu_{i}+1,\lambda_{i+1}]$.

We say that a partition $\seplistvar$ is \textit{admissible} if
every $k$ satisfies $\supp(i_k) \neq \varnothing$. (This is
clearly satisfied when there exist $j_1<j_2<\cdots<j_s$
satisfying (\ref{cond:necessary}), but also in other cases.)
Assume that $\seplistvar = \left(i_1 \geq \cdots \geq i_s > 0\right)$ is an admissible partition.
For two integers $a< b$, we let $\nuab$ denote the subpartition $(i_r,i_{r+1},\dots,i_{r+q})$ of $\seplistvar$, where $[r, r+q]$ is the (possibly empty) set of all $k$ for which $\supp(i_k)\subseteq [a,b)$. In this case, we put\footnote{Here and in the following, $\#\kappa$ denotes the length of a partition $\kappa$.} $\#\nuab:=q+1$, which is just the number of entries in $\nuab$. Similarly, we set $\nuabcap$ to be the subpartition $(i_r,i_{r+1},\dots,i_{r+q})$ of $\seplistvar$, where $[r, r+q]$ is the set of all $k$ for which $\supp(i_k)\cap [a,b)\neq\emptyset$.
For example, for $\seplistvar=(4,3,3,2)$ and $\lm$ as on Figure \ref{fig:seplist}, we have 
$$\supp(3)=[3,4],\ \supp(2)=[4,7],\ \supp(4)=[1,4],$$ 
$$\nuxy{2}{7}=(3,3),\ \nuxy{2}{8}=(3,3,2),\ \nuxy{4}{8}=(2),\ \nuxycap{4}{5}=(4,3,3,2),\ \#\nuxy{2}{7}=2.$$

\begin{remark}
 If $\seplistvar$ is not admissible, that is, if $\supp(i_k)=\emptyset$ for some $k$, then $i_k$ belongs to $\nuab$ for any $a,b$, so $\nuab$ might no longer be a contiguous subpartition of $\seplistvar$. On the other hand, if $\seplistvar$ is an admissible partition, then the partitions $\nuab$ and $\nuabcap$ are clearly admissible as well. For the rest of this section, we will only work with admissible partitions.
\end{remark}

We introduce several definitions: An admissible partition $\seplistvar=(i_1\geq \cdots\geq i_s>0)$ is called

\begin{tabular}{@{$\bullet$ }lll}
 \textit{non-representable} & if for some $a<b$ we have & $\#\nuab>b-a$;\\
 \textit{representable}& if for all $a<b$ we have& $\#\nuab\leq b-a$;\\
\end{tabular}

a representable partition $\seplistvar$ is called

\begin{tabular}{@{$\bullet$ }lll} 
 \textit{irreducible}& if for all  $a<b$ we have &$\#\nuab < b-a$;\\
 \textit{reducible} &if for some $a<b$ we have &$\#\nuab=b-a$.
\end{tabular}

For example, $\seplistvar=(4,3,3,2)$ is representable but reducible because we have $\nuxy{3}{5}=(3,3)$ so $\#\nuxy{3}{5}=2=5-3$.

Note that these notions depend on the skew partition; thus, when we want to use a skew partition $\widetilde\lm$ rather than $\lm$, we will write that $\seplistvar$ is non-representable/irreducible/etc. \textit{with respect to $\widetilde\lm$}, and we denote the corresponding partitions by $\nuab^{\widetilde{\lm}}$.

These definitions can be motivated as follows. Suppose that a partition $\seplistvar$ is non-representable, so there exist integers $a<b$ such that $\#\nuab>b-a$. Recall that $\nuab=:(i_r,i_{r+1},\dots,i_{r+q})$ contains all entries of $\seplistvar$ whose support is a subset of $[a,b)$. Thus in order for condition (\ref{cond:necessary}) to be true there must exist some integers $j_r<j_{r+1}<\cdots<j_{r+q}$ such that 
$$(i_r, j_r),(i_r+1, j_r),\dots,(i_{r+q},j_{r+q}),(i_{r+q}+1,j_{r+q})\in\lm.$$
On the other hand, by the definition of the support, we must have $j_k\in \supp(i_k)\subseteq [a,b)$ for all $r\leq k\leq r+q$. Therefore we get $q+1$ distinct elements of $[a,b)$ which is impossible if $q+1=\#\nuab>b-a$. It means that a non-representable partition $\seplistvar$ is never a seplist-partition of a 12-rpp $T$. 

Suppose now that a partition $\seplistvar$ is reducible, so for some $a<b$ we get an equality $\#\nuab=b-a$. Then these integers $j_r<\cdots<j_{r+q}$ should still all belong to $[a,b)$ and there are exactly $b-a$ of them, hence 
\begin{equation}
j_r=a,\ j_{r+1}=a+1,\ \cdots,\ j_{r+q}=a+q=b-1 .
\label{eq.reducible.mixblock-j}
\end{equation}
Because $\supp(i_r)\subseteq[a,b)$ but $\supp(i_r)\neq\varnothing$ (since $\nu$ is admissible), we have $(i_r,a-1)\notin\lm$. Thus, placing a $1$ into $(i_r,a)$ and $2$'s into $(i_r+1,a),(i_r+2,a),\dots$ does not put any restrictions on entries in columns $1,\dots, a-1$. And the same is true for columns $b,b+1,\dots$ when we place a $2$ into $(i_{r+q}+1,b-1)$ and $1$'s into all cells above. Thus, if a partition $\seplistvar$ is reducible, then the filling of columns $a,a+1,\dots, b-1$ is uniquely determined (by (\ref{eq.reducible.mixblock-j})), and the filling of the rest can be arbitrary -- the problem of existence of a 12-rpp $T$ such that $\seplist(T)=\seplistvar$ reduces to two smaller independent problems of the same kind (one for the columns $1,2,\ldots,a-1$, the other for the columns\footnote{Recall that a 12-rpp of shape $\lm$ cannot have any nonempty column beyond the $\lambda_1$'th one.} $b,b+1,\ldots,\lambda_1$). One can continue this reduction process and end up with several independent irreducible components separated from each other by mixed columns. An illustration of this phenomenon can be seen on Figure \ref{fig:seplist}: the columns $3$ and $4$ must be mixed for any 12-rpps $T$ with $\seplist(T)=(4,3,3,2)$.

More explicitly, we have thus shown that every nonempty interval $\left[a,b\right) \subseteq \left[1, \lambda_1+1\right)$ satisfying $\#\nuab = b-a$ splits our problem into two independent subproblems. But if two such intervals $\left[a,b\right)$ and $\left[c,d\right)$ satisfy $a\leq c\leq b\leq d$ then their union $[a,d)$ is another such interval (because in this case, inclusion-exclusion gives $\#\nuxy{a}{d} \geq \#\nuab + \#\nuxy{c}{d} - \#\nuxy{c}{b}$, but $\#\nuxy{c}{b}\leq b-c$ by representability of $\nu$).
Hence, the maximal (with respect to inclusion) among all such intervals are pairwise disjoint and separated from each other by at least a distance of $1$.
This yields part \textbf{(a)} of the following lemma:
\begin{lemma}
\label{lemma:irreducible}
 Let $\seplistvar$ be a representable partition.

\begin{enumerate}

\item[\textbf{(a)}] There exist unique integers $(1=b_0\leq a_1<b_1<a_2<b_2<\cdots<a_r<b_r\leq a_{r+1}=\lambda_1+1)$ satisfying the following two conditions:
 \begin{enumerate}
  \item For all $1\leq k\leq r$, we have $\#\nuxy{a_k}{b_k} =b_k-a_k$.
  \item The set $\bigcup_{k=0}^{r}[b_k,a_{k+1})$ is minimal (with respect to inclusion) among all sequences $(1=b_0\leq a_1<b_1<a_2<b_2<\cdots<a_r<b_r\leq a_{r+1}=\lambda_1+1)$ satisfying property 1.
 \end{enumerate}

Furthermore, for these integers, we have:

\item[\textbf{(b)}] The partition $\nu$ is the concatenation
\[
\left(\nuxycap{b_0}{a_1}\right) \left(\nuxy{a_1}{b_1}\right) \left(\nuxycap{b_1}{a_2}\right) \left(\nuxy{a_2}{b_2}\right) \cdots \left(\nuxycap{b_r}{a_{r+1}}\right)
\]
(where we regard a partition as a sequence of positive integers, with no trailing zeroes).

\item[\textbf{(c)}] The partitions $\nuxycap{b_k}{a_{k+1}}$ are irreducible with respect to $\lm\big|_{[b_k,a_{k+1})}$,
 which is the skew partition $\lm$ with columns $1,2,\dots,b_k-1,a_{k+1},a_{k+1}+1,\dots$ removed.

\end{enumerate}
\end{lemma}

\begin{proof}
Part \textbf{(a)} has already been proven.

\textbf{(b)} Let $\nu=(i_1\geq\cdots\geq i_s>0)$. If $\supp(i_r)\subseteq [a_k,b_k)$ for some $k$, then $i_r$ appears in exactly one of the concatenated partitions, namely, $\nuxy{a_k}{b_k}$. Otherwise there is an integer $k$ such that $\supp(i_r)\cap [b_k,a_{k+1})\neq\emptyset$. It remains to show that such $k$ is unique, that is, that $\supp(i_r)\cap [b_{k+1},a_{k+2})=\emptyset$. Assume the contrary. The interval $[a_{k+1},b_{k+1})$ is nonempty, therefore there is an entry $i$ of $\nu$ with $\supp(i)\subseteq [a_{k+1},b_{k+1})$. It remains to note that we get a contradiction: we get two numbers $i,i_r$ with $\supp(i_r)$ being both to the left and to the right of $\supp(i)$. 

\textbf{(c)} Fix $k$. Let $J$ denote the restricted skew partition $\lm\big|_{[b_k,a_{k+1})}$, and let $\seplistvar^{\prime} = \nuxycap{b_k}{a_{k+1}}$.
We need to show that if $\left[c,d\right)$ is a nonempty interval contained in $\left[b_k, a_{k+1}\right)$, then $\#\nupxy{c}{d}^{J} < d-c$. We are in one of the following four cases:

\begin{itemize}
\item \textit{Case 1: We have $c> b_k$ (or $k=0$) and $d< a_{k+1}$ (or $k=r$).} In this case, every $i_p$ with $\supp^J(i_p) \subseteq [c,d)$ must satisfy $\supp(i_p) \subseteq [c,d)$. Hence, $\nupxy{c}{d}^{J} = \nuxy{c}{d}$, so that $\#\nupxy{c}{d}^{J} = \#\nuxy{c}{d} < d-c$, and we are done.
\item \textit{Case 2: We have $c=b_k$ and $k>0$ (but not $d=a_{k+1}$ and $k<r$).} Assume (for the sake of contradiction) that $\#\nupxy{c}{d}^{J} \geq d-c$. Then, the $i_p$ satisfying $\supp^J(i_p) \subseteq [c,d)$ must satisfy $\supp(i_p) \subseteq [a_k, d)$ (since otherwise, $\supp(i_p)$ would intersect both $[b_{k-1}, a_k)$ and $[b_k, a_{k+1})$, something we have ruled out in the proof of \textbf{(b)}). Thus, $\#\nuxy{a_k}{d} \geq (d-c)+(b_k-a_k) = d-a_k$, which contradicts the minimality of $\bigcup_{k=0}^{r}[b_k,a_{k+1})$ (we could increase $b_k$ to $d$).
\item \textit{Case 3: We have $d=a_{k+1}$ and $k<r$ (but not $c=b_k$ and $k>0$).} The argument here is analogous to Case 2.
\item \textit{Case 4: Neither of the above.} Exercise.
\end{itemize}
\end{proof}

\begin{definition}
 
 In the context of Lemma \ref{lemma:irreducible}, for $0\leq k\leq r$ the subpartitions $\nuxycap{b_k}{a_{k+1}}$ are called \textit{the irreducible components of $\seplistvar$} and the nonnegative integers $n_k:=a_{k+1}-b_k-\#\nuxycap{b_k}{a_{k+1}}$ are called their \textit{degrees}. (For $T$ with $\seplist(T)=\seplistvar$, the $k$-th degree $n_k$ is equal to the number of pure columns of $T$ inside the corresponding $k$-th irreducible component. All $n_k$ are positive, except for $n_0$ if $a_1 = 1$ and $n_r$ if $b_r = \lambda_1 + 1$.)
\end{definition}

\begin{example}
 For $\seplistvar=(4,3,3,2)$ we have $r=1,b_0=1,a_1=3,b_1=5,a_2=8$. The irreducible components of $\seplistvar$ are $(4)$ and $(2)$ and their degrees are $3-1-1=1$ and $8-5-1=2$ respectively. We have $\nuxycap{1}{3}=(4),\nuxy{3}{5}=(3,3),\nuxycap{5}{8}=(2)$.
\end{example}

\subsection{The structural theorem and its applications}
It is easy to see that for a 12-rpp $T$, the number $\#\seplist(T)$ is equal to the number of mixed columns in $T$.

Recall that $\OneTwoRPP$ denotes the set of all 12-rpps $T$ of shape $\lm$, and let $\OneTwoRPPCutvar$ denote its subset consisting of all 12-rpps $T$ with $\seplist(T)=\seplistvar$. Now we are ready to state a theorem that completely describes the structure of irreducible components (which will be proven later):
\begin{theorem}
\label{thm:12rpps}
 Let $\seplistvar$ be an irreducible partition. Then for all $0\leq m\leq \lambda_1-\#\seplistvar$ there is exactly one 12-rpp $T\in\OneTwoRPPCutvar$ with $\#\seplistvar$ mixed columns, $m$ $1$-pure columns and $(\lambda_1-\#\seplistvar-m)$ $2$-pure columns. Moreover, these are the only elements of $\OneTwoRPPCutvar$. In other words, for an irreducible partition $\seplistvar$ we have
  \begin{equation}
 \label{eq:thm12rpps}
 \sum_{T\in\OneTwoRPPCutvar} \x^{\ircont(T)}=(x_1x_2)^{\#\seplistvar} P_{\lambda_1-\#\seplistvar}(x_1,x_2).
 \end{equation}
\end{theorem}

\begin{example}
 Each of the two 12-rpps on Figure \ref{fig:seplist} has two irreducible components. One of them is supported on the first two columns and the other one is supported on the last three columns. Here are all possible 12-rpps for each component:
 
\begin{tabular}{c|c}
\begin{tabular}{cc}
 & \\
\begin{ytableau}
1 & \one \\
1 & 2
\end{ytableau}\ &
\begin{ytableau}
\one& 2\\
2   & 2
\end{ytableau}\\
\end{tabular} & 
\begin{tabular}{ccc}
\begin{ytableau}
\none & 1 & 1\\
1 & 1 & \one \\
1 & 1 & 2
\end{ytableau}\ &
\begin{ytableau}
\none & 1 & 2\\
1 & \one & 2 \\
1 & 2 & 2
\end{ytableau}\ &
\begin{ytableau}
\none & 2 & 2\\
\one & 2 & 2 \\
2 & 2 & 2
\end{ytableau} \\
\end{tabular}\\
$\lambda=(2,2);\ \mu=();\ \seplistvar=(4)$ & $\lambda=(3,3,3);\ \mu=(1);\ \seplistvar=(2)$. 
\end{tabular}\\
\end{example}

After decomposing into irreducible components, we can obtain a formula for general representable partitions:
\begin{corollary}
\label{cor:thm12rppsCompositeCor}
 Let $\seplistvar$ be a representable partition. Then
  \begin{equation}
  \label{eq:thm12rppsComposite}
\sum_{T\in\OneTwoRPPCutvar} \x^{\ircont(T)}=(x_1x_2)^{M} P_{n_0}(x_1,x_2)P_{n_1}(x_1,x_2)\cdots P_{n_r}(x_1,x_2),
  \end{equation}
 where the numbers $M,r,n_0,\dots,n_r$ are defined above: $M=\#\seplistvar$, $r+1$ is the number of irreducible components of $\seplistvar$ and $n_0,n_1,\dots,n_r$ are their degrees.
\end{corollary}

\begin{proof}[Proof of Corollary \ref{cor:thm12rppsCompositeCor}.] The restriction map
\[
\OneTwoRPPCutvar \to \prod_{k=0}^r \operatorname{RPP}^{12}\left( \lm\big|_{[b_k,a_{k+1})} ; \nuxycap{b_k}{a_{k+1}}\right)
\]
is injective (since, as we know, the entries of a $T \in \OneTwoRPPCutvar$ in any column outside of the irreducible components are uniquely determined) and surjective (as one can ``glue'' rpps together).
Now use Theorem \ref{thm:12rpps}.
\end{proof}

 For a 12-rpp $T$, the vectors $\seplist(T)$ and $\ceq(T)$ uniquely determine each other: if $(\ceq(T))_i=h$ then $\seplist(T)$ contains exactly $\lambda_{i+1}-\mu_i-h$ entries equal to $i$, and this correspondence is one-to-one. Therefore, the polynomials on both sides of (\ref{eq:thm12rppsComposite}) are equal to $Q_{\ceqvar}(x_1,x_2)$ where the vector $\ceqvar$ is the one that corresponds to $\seplistvar$.

 Note that the polynomials $P_n(x_1,x_2)$ are symmetric for all $n$. Since the question about the symmetry of $\g_\lm$ can be reduced to the two-variable case, Corollary \ref{cor:thm12rppsCompositeCor} gives an alternative proof of the symmetry of $\g_\lm$:
\begin{corollary}
 The polynomials $\g_\lm \in \Z[t_1,t_2,\dots]\left[[x_1, x_2, x_3, \ldots]\right]$ are symmetric.
\end{corollary}
 
 Of course, our standing assumption that $\lm$ is connected can be lifted here, because in general, $\g_\lm$ is the product of the analogous power series corresponding to the connected components of $\lm$. So we have obtained a new proof of Theorem \ref{thm.gtilde.symm}.
 

Another application of Theorem \ref{thm:12rpps} is a complete description of Bender-Knuth involutions on rpps.

\begin{corollary}
\label{cor:uniqueBK}
Let $\seplistvar$ be an irreducible partition. Then there is a unique map $b:\OneTwoRPPCutvar\to\OneTwoRPPCutvar$ such that for all $T\in\OneTwoRPPCutvar$, the sequence $\ircont(b(T))$ is obtained from $\ircont(T)$ by switching the first two entries.
This unique map $b$ is an involution on $\OneTwoRPPCutvar$. So, for irreducible partition $\seplistvar$ the corresponding Bender-Knuth involution exists and is unique.
\end{corollary}

Take any 12-rpp $T\in\OneTwoRPPCutvar$ and recall that a 12-table $\flip(T)$ is obtained from $T$ by simultaneously replacing all entries in $1$-pure columns by $2$ and all entries in $2$-pure columns by $1$. 

\begin{corollary}
\label{cor:confluence}
 If $\seplistvar$ is an irreducible partition, then, no matter in which order one resolves descents in $\flip(T)$, the resulting 12-rpp $T'$ will be the same. The map $T\mapsto T'$ is the unique Bender-Knuth involution on $\OneTwoRPPCutvar$.
\end{corollary}
\begin{proof}[Proof of Corollary \ref{cor:confluence}]
 Descent-resolution steps applied to $\flip(T)$ in any order eventually give an element of $\OneTwoRPPCutvar$ with the desired $\ircont$. There is only one such element. So we get a map $\OneTwoRPPCutvar\to \OneTwoRPPCutvar$ that satisfies the assumptions of Corollary \ref{cor:uniqueBK}.
\end{proof}

Finally, notice that, for a general representable partition $\seplistvar$, descents in a 12-table $T$ with $\seplist(T) = \nu$ may only occur inside each irreducible component independently. Thus, we conclude the chain of corollaries by stating that our constructed involutions are canonical in the following sense:

\begin{corollary}
 For a representable partition $\nu$, the map $\mathbf{B}:\OneTwoRPPCutvar\to\OneTwoRPPCutvar$ is the unique involution that interchanges the number of $1$-pure columns with the number of $2$-pure columns inside each irreducible component.
\end{corollary}

\subsection{The proof}
Let $\seplistvar=(i_1,\dots,i_s)$ be an irreducible partition. We start with the following simple observation:
\begin{lemma}
\label{lemma:leftRight}
 Let $T\in\OneTwoRPPCutvar$ for an irreducible partition $\seplistvar$. Then any $1$-pure column of $T$ is to the left of any $2$-pure column of $T$.
\end{lemma}
\begin{proof}[Proof of Lemma \ref{lemma:leftRight}]
Suppose it is false and we have a $1$-pure column to the right of a $2$-pure column. Among all pairs $\left(a, b\right)$ such that column $a$ is $2$-pure and column $b$ is $1$-pure, and $b > a$, consider the one with smallest $b-a$. Then, the columns $a+1,\dots,b-1$ must all be mixed. 
Therefore the set $\NS(T)$ contains $\{(i_{p+1},a+1),(i_{p+2},a+2),\dots,(i_{p+b-1-a},b-1)\}$ for some $p\in\NN$. And because $a$ is $2$-pure and $b$ is $1$-pure, each $i_{p+k}$ (for $k=1,\dots, b-1-a$) must be $\leq$ to the y-coordinate of the highest cell in column $a$ and $>$ than the y-coordinate of the lowest cell in column $b$. Thus, the support of any $i_{p+k}$ for $k=1,\dots, b-1-a$ is a subset of $[a+1,b)$, which contradicts the irreducibility of $\seplistvar$.
\end{proof}

\begin{proof}[Proof of Theorem \ref{thm:12rpps}]
We proceed by strong induction on the number of columns in $\lm$. If the number of columns is $1$, then the statement of Theorem \ref{thm:12rpps} is obvious. Suppose that we have proven that for all skew partitions $\widetilde{\lm}$ with less than $\lambda_1$ columns and for all partitions $\widetilde\seplistvar$  irreducible with respect to $\widetilde{\lm}$ and for all $0\leq \widetilde m\leq \widetilde{\lambda}_1-\#\widetilde\seplistvar$, there is exactly one 12-rpp $\widetilde T$ of shape $\widetilde{\lm}$ with exactly $\widetilde m$ $1$-pure columns, exactly $\#\widetilde\seplistvar$ mixed columns and exactly $(\widetilde{\lambda}_1-\#\widetilde\seplistvar-\widetilde m)$ $2$-pure columns. Now we want to prove the same for $\lm$.

Take any 12-rpp $T\in\OneTwoRPPCutvar$ with $\seplist(T)=\seplistvar$ and with $m$ $1$-pure columns for $0\leq m\leq \lambda_1-\#\seplistvar$. Suppose first that $m>0$. Then there is at least one $1$-pure column in $T$. Let $q\geq 0$ be such that the leftmost $1$-pure column is column $q+1$. Then by Lemma \ref{lemma:leftRight} the columns $1,2,\dots,q$ are mixed. If $q>0$ then the supports of $i_1,i_2,\dots,i_q$ are all contained inside $[1,q+1)$ and we get a contradiction with the irreducibility of $\seplistvar$. The only remaining case is that $q=0$ and the first column of $T$ is $1$-pure. Let $\widetilde\lm$ denote $\lm$ with the first column removed. Then $\seplistvar$ is obviously admissible but may not be irreducible with respect to $\widetilde{\lm}$, because it may happen that $\#\nuxy{2}{b+1}^{\widetilde{\lm}}=b-1$
for some $b>1$. In this case we can remove these $b-1$ nonempty columns from $\widetilde{\lm}$ and remove the first $b-1$ entries from $\nu$ to get an irreducible partition again\footnote{This follows from Lemma \ref{lemma:irreducible} \textbf{(c)} (applied to the skew shape $\widetilde{\lm}$ and $k=1$). Here we are using the fact that if we apply Lemma \ref{lemma:irreducible} \textbf{(a)} to $\widetilde{\lm}$ instead of $\lm$, then we get $r = 1$ (because if $r \geq 2$, then $\#\nuxy{a_2}{b_2} = \#\nuxy{a_2}{b_2}^{\widetilde{\lm}} = b_2-a_2$ in contradiction to the irreducibility of $\lm$).}, to which we can apply the induction hypothesis. We are done with the case $m>0$. If $m<\lambda_1-\#\seplistvar$ then we can apply a mirrored argument to the last column, and it remains to note that the cases $m>0$ and $m<\lambda_1-\#\seplistvar$ cover everything (since the irreducibility of $\nu$ shows that $\lambda_1-\#\seplistvar>0$).

This inductive proof shows the uniqueness of the 12-rpp with desired properties. Its existence follows from a parallel argument, using the observation that the first $b-1$ columns of $\widetilde{\lm}$ can actually be filled in. This amounts to showing that for a representable $\seplistvar$, the set $\OneTwoRPPCutvar$ is non-empty in the case when $\lambda_1 = \#\seplistvar$ (so all columns of $T \in \OneTwoRPPCutvar$ must be mixed). This is clear when there is just one column, and the general case easily follows by induction on the number of columns\footnote{In more detail: \par
If we had $1 \notin \supp(\seplistvar_1)$, then we would have $\supp(\seplistvar_1) \subseteq [2, \lambda_1+1)$, and thus $\supp(\seplistvar_j) \subseteq [2, \lambda_1+1)$ for every $j$ (since $\seplistvar$ is weakly decreasing and since $\supp(\seplistvar_1)$ is nonempty), which would lead to $\nuxy{2}{\lambda_1+1} = \seplistvar$ and thus $\#\nuxy{2}{\lambda_1+1} = \#\seplistvar = \lambda_1 > \lambda_1 + 1 - 2$, contradicting the representability of $\seplistvar$. Hence, we have $1 \in \supp(\seplistvar_1)$, so that we can fill the first column of $\lm$ with $1$'s and $2$'s in such a way that it becomes mixed and the $1$'s are displaced by $2$'s at level $\seplistvar_1$. Now, let $\widetilde{\lm}$ be the skew partition $\lm$ without its first column, and $\widetilde{\seplistvar}$ be the partition $\left(\seplistvar_2, \seplistvar_3, \ldots\right)$. Then, the partition $\widetilde{\seplistvar}$ is representable with respect to $\widetilde{\lm}$. (Otherwise we would have $\#\nuxy{2}{b+1}^{\widetilde{\lm}}>b-1$ for some $b \geq 1$, but then we would have $\supp(\seplistvar_1) \subseteq [1,b+1)$ as well and therefore $\#\nuxy{1}{b+1}>(b-1)+1=b$, contradicting the representability of $\lm$.) Thus we can fill in the entries in the cells of $\widetilde{\lm}$ by induction.}.
\end{proof}

\end{document}